\documentclass[11pt]{article}

\usepackage{amsmath, amsthm, amsopn, amssymb, enumerate,hyperref}

\newtheorem{thm}{Theorem}
\newtheorem{theorem}{Theorem}
\newtheorem{lemma}[thm]{Lemma}
\newtheorem{prop}[thm]{Proposition}
\newtheorem{cor}[thm]{Corollary}
\newtheorem{claim}[thm]{Claim}

\newtheorem*{con}{Construction}
\newtheorem*{notation}{Notation}

\numberwithin{equation}{section}
\numberwithin{thm}{section}

\theoremstyle{definition}

\newcommand{\g}{\mathcal{G}}
\newcommand{\SL}{\mathsf{SL}}
\newcommand{\eps}{\varepsilon}
\newcommand{\mat}[4]{\left( \begin{array}{cc} #1 & #2 \\ #3 & #4 \\ \end{array} \right)}
\newcommand{\norm}[1]{\left\| #1 \right\|}
\newcommand{\Z}{\mathbb{Z}}
\newcommand{\R}{\mathbb{R}}
\newcommand{\ol}{\overline}
\newcommand{\ip}[2]{\left\langle #1 , #2 \right\rangle}

\newcommand{\n}{\mathcal{N}}
\newcommand{\e}{\mathcal{E}}

\newcommand{\Tr}{\mathsf{Tr} }
\newcommand{\dist}{\mathsf{dist}}
\newcommand{\C}{\mathbb{C}}
\newcommand{\SU}{\mathsf{SU}(2)}
\newcommand{\supp}{\mathop \mathsf{supp}}
\newcommand{\re}{\mathop \mathsf{real}}

\newcommand{\sd}{\mathsf{Diag}}
\newcommand{\f}{\mathcal{F}}
\newcommand{\wh}{\widehat}
\newcommand{\K}{\mathcal{K}}
\newcommand{\w}{\mathcal{W}}

\title{Monotone expansion}

\author{Jean Bourgain\thanks{
Institute for Advanced Study, Princeton NJ, {\tt bourgain@math.ias.edu}.}  \and 
Amir Yehudayoff\thanks{
Technion--IIT, Haifa, Israel, {\tt amir.yehudayoff@gmail.com}.}}

\begin{document}

\date{}

\maketitle

\begin{abstract}
This work, following the outline set in \cite{BouMonEx}, 
presents an explicit construction of a family of monotone expanders.
The family is essentially defined by the M\"{o}bius action of
$\SL_2(\R)$ on the real line.
For the proof, we show a product-growth theorem for $\SL_2(\R)$.
\end{abstract}

\setcounter{section}{-1}

\section{Introduction}

Expanders are sparse graphs with ``strong connectivity'' properties.
Such graphs are extremely useful in various applications
(see the survey \cite{HLW}).
Most sparse graphs are expanders,
but for applications explicit constructions are needed.
Indeed, explicit constructions of expanders graphs are known, e.g. \cite{LPS,RVW}.
Here we describe an explicit construction of monotone expanders
(for more on such expanders see \cite{DW}).

The construction of monotone expander we present
first builds a ``continuous'' expander,
which in turn can be discretized to the required size.
A {\em continuous} expander is 
a finite family of maps $\Psi$ for which  
there exists a constant $c_0 > 0$ so that the following holds.
Every $\psi \in \Psi$ is a smooth 
map from the interval $[0,1]$ to itself,
and for all measurable $A \subset [0,1]$ with $|A| \leq 1/2$,
$$\left| \Psi(A) \right| \geq (1+c_0) |A| ,$$
where
$\Psi(A) = \bigcup_{\psi \in \Psi} \psi(A)$.
We say that $\Psi$ is a continuous {\em monotone} expander if in addition
every $\psi \in \Psi$ is monotone, i.e., $\psi(x) > \psi(y)$ for $x > y$.

\setcounter{theorem}{-1}
\begin{theorem}
\label{thm: mon exp intro}
There exists an explicit continuous monotone expander.
\end{theorem}

The word explicit in the theorem can be interpreted as follows.
The family $\Psi$ can be (uniformly) described by a constant number of bits,
and given a rational $x \in [0,1]$ that can be described by
$b$ bits,  
$\psi(x)$ is rational and can be computed in time
polynomial in $b$, for all $\psi \in \Psi$.

The family $\Psi$ also satisfies  
$$\norm{\psi-\mathsf{id}}_\infty ,\norm{\psi' -1}_\infty \leq c,$$ 
for every $\psi \in \Psi$,
for a small constant $c>0$, where $\mathsf{id}$ is the identity map.

The proof of the theorem follows the outline described in \cite{BouMonEx},
which in turn uses ideas from recent works on growth and expansion in matrix groups.
Most relevant is the work of Bourgain and Gamburd \cite{BG1}
showing expansion in $\SU$.
Also related, 
is the work of Bourgain and Gamburd \cite{BG2}
proving expansion in $\SL_2(\mathbb{F}_p)$,
and the work of Helfgott \cite{H} 
showing growth in $\SL_2(\mathbb{F}_p)$.

The theorem describes the existence of a continuous monotone expander.
By partitioning $[0,1]$ to $n$ equal-length intervals,
$\Psi$ naturally defines a discrete bi-partite monotone expander 
on $2n$ vertices. Namely,
a bi-partite graph $G$ with two color classes $L,R$ of size $n$ each so that
(i) for every $A \subset L$ of size $|A| \leq n/2$,
the size of $B = \{b \in R : (a,b) \in E(G)  \ \text{for some} \ a \in A \}$
is at least $(1+c) |A|$, $c>0$ a constant independent of $n$, and 
(ii) the edges $E(G)$ can be partitioned to finitely many sets 
$E_1,\ldots,E_k$, $k$ independent of $n$, 
so that in each $E_i$ edges do not ``cross'' each other
(viz., $E_i$ defines a partial monotone map).
Since $\Psi$ is explicit,
the graph $G$ is explicit as well.
(If $\Psi$ was continuous but not monotone,
the same reduction would yield
a family of discrete bi-partite expanders.)

No other proof of existence of discrete monotone expanders is known,
not even using the probabilistic method.
A partial explanation to that is the following.
Natural probability distributions on partial monotone functions
give, w.h.p., functions that are ``close'' to affine.
Klawe, however, showed in \cite{K} that if one tries to construct expanders
using affine transformations, then the minimal number of generators
required is super-constant (in the number of vertices),
and so no construction ``that is close to affine'' can work.
Two more related comments:
(i) The construction in this text uses ``generators'' that are defined as
the ratio of two affine transformations.
(ii) Dvir and Wigderson \cite{DW} showed that any proof of existence
of a family of monotone expanders yields an explicit
construction of monotone expanders.

Implicit in the work of Dvir and Shpilka \cite{DS} it is shown that
an explicit discrete monotone expander easily yields
an explicit \emph{dimension expander}.  Specifically, 
the existence of a constant number of $n \times n$ zero-one matrices
$M_1,\ldots,M_k$ so that for every field $\mathbb{F}$
and for every subspace $V$ of $\mathbb{F}^n$
of dimension $D \leq n/2$, the dimension of the span of
$M_1(V) \cup \ldots \cup M_k(V)$ is at least $(1+c) D$.
The work of Lubotzky and Zelmanov \cite{LZ}
shows that over the real numbers
any explicit (perhaps non-monotone) expander yields
an explicit dimension expander.

Here is an outline of the proof.
To present the main ideas,
we ignore many of the problematic issues.

{\em Defining maps.}  Every matrix $g \in \SL_2(\R)$ acts
on $\R$ in a monotone way via the M\"{o}bius action.  
The maps $\Psi$ will be defined
by the actions of a set of matrices $\g \subset \SL_2(\R)$.  
This ensures that the maps in $\Psi$ are monotone.
Choose $\g$ as a family of matrices that freely generate a group
(with some extra properties, see Lemma~\ref{lem: free gp} for exact statement).
To find $\g$, use the strong Tits alternative of Breuillard \cite{Br},
which roughly states that in a ball of constant radius in $\SL_2(\R)$
there are elements that freely generate a group.

{\em Proving expansion.}
As in many expanders constructions,
the expansion follows by proving that the operator $T$ 
defined by $\Psi$ has a (restricted) spectral gap.
As in recent works, the spectral gap is established as follows.
Let $\nu$ be the probability distribution defined by $\Psi$.
Then, the $\ell$-fold convolution of $\nu$ with itself, $\nu^{(\ell)}$,
is flat, even for $\ell$ relatively small.
This statement implies the rapid mixing of 
the random walk defined by $\nu$,
and hence implies expansion.
The proof consists of three steps.

{\em (i) Small $\ell$.}
To show that $\nu^{(\ell)}$ is ``somewhat'' flat for small $\ell$,
use the fact that the group generated by $\g$ is free,
and Kesten's estimates for the behavior of random walks on free groups.
Roughly, as $\g$ freely generates a group, the convolution ``grows along a tree''
and hence flat.  
Here we also need to use a ``diophantine'' property of $\g$,
i.e., that elements of $\g$ have constant rational entries.

{\em (ii) Intermediate-size $\ell$.}
This is the main part of the argument.
We prove a product-growth theorem for $\SL_2(\R)$:
if $S$ is a subset of $\SL_2(\R)$ with certain properties,
then the size of $S_{(3)} = \{s_1 s_2 s_3 : s_i \in S\}$ 
is much larger than the size of $S$.
(The outline of the proof of the product theorem
appears in Section~\ref{sec: prod thm proof}.)
Such a product theorem implies that
$\nu^{(2\ell)}$ is much flatter than $\nu^{(\ell)}$,
unless it is already pretty flat.

{\em (iii) Large $\ell$.}
By steps (i) and (ii),
we can conclude that $\nu^{(\ell)}$ is pretty flat,
even for $\ell$ relatively small.
It remains to show that $\nu^{(C\ell)}$ is very flat,
for $C > 0$ a constant.
In previous works,
this last step follows Sarnak and Xue's multiplicities argument.
As $\SL_2(\R)$ is not compact,
such an argument can not be applied here.
Instead, use the subgroup structure of $\SL_2(\R)$,
or in other words the two-transitivity 
of the M\"{o}bius action.
To do so, also use knowledge of the Fourier spectrum of the set $A$.
We are able to obtain knowledge on the spectrum of $A$
by adding to $\Psi$ the translate map.
The translate map implies that, w.l.o.g., we can assume that the spectrum of $A$
does not have low frequencies.

\section{A monotone expander}

In essence, the maps defining the monotone expander
are induced by the action of $\SL_2(\R)$ on $\R$.
To find the relevant elements of $\SL_2(\R)$,
use the following lemma.
The proof of the lemma is given in Section~\ref{sec: free gp}.

\begin{lemma}
\label{lem: free gp}
There is a constant $C > 0$ so that the following holds.
For $\eps > 0$ small, there is a positive integer $Q$
and a subset $\g$ of $\SL_2(\R)$ so that
\begin{enumerate}
\item $(1/\eps)^{1/C} < Q < (1/\eps)^C$,
\item \label{itm: prop g 2} $Q < |\g|^C$,
\item elements of $\g$ freely generate group,
\item \label{itm: dionph} elements of $\g$ have entries of the form $\Z/Q$, and
\item \label{itm: g close to 1} every $g \in \g$ admits 
$$\norm{g-1}_2 
= (g_{1,1} -1 )^2 + (g_{1,2})^2 + (g_{2,1})^2 + (g_{2,2}-1)^2 \leq \eps.$$
\end{enumerate}
\end{lemma}

The lemma summarizes all the properties
$\g$ should satisfy in order to yield a monotone expander.
When applying the lemma, $\eps$ is a small universal constant.
An important (and useful) property of the lemma is that
both $|\g|$ and $Q$ are polynomially comparable to $1/\eps$.
Without this property, the lemma immediately follows from the strong Tits alternative \cite{Br}.
Property~\ref{itm: dionph} 
yields the {\em non-commutative diophantine} property of $\g$,
roughly, that for every $w \neq w'$ 
that are words of length $k$ in the element of $\g$,
the distance between $w$ and $w'$ is at least $(1/Q)^k$.
This property is defined and used in \cite{BG1}.
Property~\ref{itm: g close to 1}
is crucial for handling the non-compactness of $\SL_2(\R)$.

Consider the M\"{o}bius action: 
Given $g = \mat{a}{b}{c}{d}$ in $\SL_2(\R)$,
denote by $\ol g$ the map defined by 
$$\ol g(x) = \frac{ax + b}{cx + d} .$$
For all $g$ in $\SL_2(\R)$, the derivative of the map $\ol g$ is
$${\ol g}'(x) 
= \frac{1}{(cx + d)^2} .$$
So $\ol g$ is monotone in any interval not containing $-d/c$.

\begin{con}
Let $\Psi$ be the family of monotone smooth maps $\psi$ from sub-intervals of
$[0,1]$ to $[0,1]$ defined as follows.

Let $\eps > 0$ be a small universal constant (to be determined).
Let $\g$ be the family of matrices given by Lemma~\ref{lem: free gp}.
Define $$\Psi_\g = \{ \ol{g} : g  \in \g \cup \g^{-1}\}.$$
Here we restrict $\ol{g}$ to output values in $[0,1]$, i.e.,
if $\psi \in \Psi_\g$ is defined by $g$, 
then $\psi$ is a map from 
the interval $\ol g^{-1}([0,1]) \cap [0,1]$ to $[0,1]$.

Let $K = K(\eps)$ be a large integer (to be determined).
Define the map $\psi_+ : [0,1-1/K] \to [1/K,1]$
by $\psi_+(x) = x+1/K$,
and the map $\psi_- : [1/K,1] \to [0,1-1/K]$
by $\psi_-(x) = x-1/K$.

Finally,
$$\Psi = \Psi_\g \cup \{\psi_+,\psi_- , \mathsf{id} \}, $$
where $\mathsf{id}$ is the identity map.
\end{con}

\begin{thm}
\label{thm: mon exp}
There is a constant $c_0 > 0$ so that the following holds.
Let $A$ be a measurable subset of $[0,1]$
with $|A| \leq 1/2$.
Then, $\left| \Psi(A) \right| \geq (1+c_0) |A|$.
\end{thm}

Theorem~\ref{thm: mon exp} implies Theorem~\ref{thm: mon exp intro},
and follows from the following ``restricted spectral gap'' theorem.
(To see that $\Psi$ is explicit,
add to $\Psi$ all maps from ``the large ball''
in the proof of Lemma~\ref{lem: free gp}.)
The M\"{o}bius action induces a unitary 
representation of $\SL_2(\R)$ on $L^2(\R)$ defined by
$$T_{g^{-1}} f (x) = \sqrt{{\ol g}'}(x) f ( \ol{g}(x)). $$
For a positive integer $K$, denote by $\f_K$ the family of maps
$f \in L^2(\R)$ with $\supp(f) \subset [0,1]$ and $\norm{f}_2 = 1$ 
so that for all $k \in \{1,2,\ldots,K\}$,
$$\int_{I(k)} f(x) dx = 0,$$
where
$$I(k) = [(k-1)/K,k/K].$$

\begin{thm}
\label{thm: rest spec}
Let $\eps > 0$ be a small enough constant.
Let $\g$ be the set given by Lemma~\ref{lem: free gp}.
If $K  = K(\eps)$ is a large enough positive integer, then 
for all $f \in \f_K$,
\begin{align}
\label{eqn: Tf f is small}
\ip{\sum_{g} \nu(g) T_g f}{f} < 1/2 ,
\end{align}
with the probability measure
$$\nu = (2 |\g|)^{-1} \sum_{g \in \g} {\bf 1}_g + {\bf 1}_{g^{-1}},$$
where ${\bf 1}_g$ is the delta function at $g$.
\end{thm}

The ``restricted spectral gap'' theorem is proved in Section~\ref{sec: proof of rest spec gap}.

\begin{proof}[Proof of Theorem~\ref{thm: mon exp}]
We first reduce the general case to the ``restricted spectral gap'' case.
Let $\sigma > 0$ be a small universal constant, to be determined.
If there is $k \in \{1,\ldots,K-1\}$ so that
$$\big| |A \cap I(k+1)| - |A\cap I(k)| \big| \geq \sigma |A| ,$$
then, using the maps $\psi_+,\psi_-$ and $\mathsf{id}$,
$$\left| \Psi(A) \right|
\geq (1+ \sigma)|A|.$$

It thus remains to consider the case that 
$\big| |A \cap I(k+1)| - |A\cap I(k)| \big| < \sigma |A|$
for all $k$. Thus, for all $k$,
\begin{align}
\label{eqn: a is bal}
\big| K |A \cap I(k)| - |A| \big| < \sigma K^2 |A| .
\end{align}

Assume towards a contradiction that the theorem does not hold.

Since $\norm{g-1}_2 \leq \eps$, 
for all $x \in [0,1]$,
$$ \frac{1}{(1 +2 \eps)^2} <
{\ol g}'(x) 
< \frac{1}{(1-2 \eps)^2} . $$
Thus, for every $x \in [0,1]$, 
\begin{align*}
 0 \leq \ol g(x) - x 
 < 10 \eps .
\end{align*}
We need to ensure that even after applying the maps in $\Psi_\g$
we remain in $[0,1]$.
To this end, let
$$A' = A \cap [k'/K, 1- k'/K]$$
with $k'$ the smallest integer so that $k' \geq 10 \eps K$. 
By \eqref{eqn: a is bal}, 
$$0.99 |A| \leq |A'| \leq |A|,$$
as long as $\sigma,\eps$ are small.

Denote
$$f = {\bf 1}_{A'} - |A'|.$$
For all $g \in \g \cup \g^{-1}$, 
$$\ip{T_{g-1} f}{f} 
\geq \frac{1}{1-7\eps} \int ({\bf 1}_{A'}(\ol{g}(x)) - |A'| )
({\bf 1}_{A'}(x) - |A'|) dx \geq 0.9 |A'|(1-|A'|)
\geq 0.8 \norm{F}_2,$$
as long as $\sigma,\eps,c_0$ are small.

Project $A'$ on $\f_K$.
Define $F$ as follows: for all $x \in [0,1]$,
if $x \in I(k)$, then
$$F(x) = {\bf 1}_{A'}(x) - K |A' \cap I(k)| .$$
Hence, $F/ \norm{F}_2 \in \f_K$.
In addition, for $\sigma$ small, using \eqref{eqn: a is bal},
$$\norm{f - F}_2^2 = \sum_{k = k'}^{K-k'}
\int_{I(k)} (|A'| - K |A' \cap I(k)|)^2 dx
\leq 2 \sigma^2 K^4 |A'|^2 \leq  0.01 \norm{F}^2_2 .$$
Therefore, 
\begin{align*}
0.8 \norm{F}_2^2 
& \leq \ip{\sum_g \nu(g) T_g ( f - F + F)}{f- F + F} 
 \leq 0.1 \norm{F}^2_2 
+\ip{\sum_g \nu(g) T_g F}{F} ,
\end{align*}
which 
contradicts Theorem~\ref{thm: rest spec}.
\end{proof}

\section{Finding set of generators}
\label{sec: free gp}

\begin{notation}
For convenience, we use the following notation throughout the text.
For a constant $c \in \R$, we denote by
$c+$ a constant slightly larger than $c$,
and by $c-$ a constant slightly smaller than $c$.
Typically, the meaning of ``slightly'' depends
on other parameters that are clear from the context.
We also use the following asymptotic notation.
Write $a \lesssim b$ if $a \leq Cb$ with $C$ a universal constant.
Write $a \gtrsim b$ if $b \lesssim a$,
and $a \sim b$ if $a \lesssim b \lesssim a$.
\end{notation}

For $\delta > 0$, 
denote by $B_\delta(x)$ the ball of radius $\delta$ around $x$
and by $\Gamma_{\delta}(A)$ the $\delta$-neighborhood of the set $A$.
We consider the $L^2$-metric on $\SL_2(\R)$.

\begin{proof}[Proof of Lemma~\ref{lem: free gp}]
Breuillard \cite{Br} proved a strong Tits alternative:
there is a constant $r \in \Z$
so that if $S$ is a finite symmetric subset of $\SL_2(\R)$, which
generates a non-amenable subgroup, then $S_{(r)}
= \{s_1 s_2 \cdots s_r : s_i \in S\}$ contains two elements,
which freely generate a group.

Let 
$$h_1 = \mat{1}{1/q}{0}{1} 
\ \ \text{and} \ \ 
h_2 = \mat{1}{0}{1/q}{1}.$$
Observe 
$$h_1^q = \mat{1}{1}{0}{1} 
\ \ \text{and} \ \
h_2^q = \mat{1}{0}{1}{1}.$$
Hence, $h_1,h_2$ generate a non-amenable group.
Apply the strong Tits alternative on the set
$S = \{h_1,h_2,h_1^{-1},h_2^{-1}\}$.
There are thus $g_1,g_2 \in S_{(r)}$ that freely generate a group.

It remains to convert $g_1,g_2$ to many elements that are close to identity
and freely generate a group.
Let $\ell \sim \log (1/\eps)$ so that the following holds.
Consider
$$W = \left\{ w^2 :
w = s_1 \cdots s_\ell , \
s_1 = g_1, \ s_\ell = g_2 , \ 
 s_i \in \{g_1,g_2,g_1^{-1},g_2^{-1}\} , \ s_{i+1} \neq s_i^{-1} \right\}.$$
Say that a word $\sigma_1 \sigma_2 \cdots \sigma_k$ in 
an alphabet $\Sigma \cup \Sigma^{-1}$
is $\langle \Sigma \rangle${\em-reduced}
if $\sigma_{i+1} \neq \sigma^{-1}_i$ for all $i \in \{1,\ldots,k-1\}$.
The size of $W$ is order $3^{\ell}$
and $W$ consists of words of $\langle g_1,g_2 \rangle$-reduced-length exactly $2\ell$.

\begin{claim}
The elements of $W$ freely generate a group.
\end{claim}

\begin{proof}
Let $w_1 \neq w_2^{-1}$ in $W \cup W^{-1}$.
Write
$$w_1 = (g_{a_1} s_1 g_{b_1})^2
\ \ \text{and} \ \
w_2 = (g_{a_2} s_2 g_{b_2})^2$$
with $s_1,s_2$ reduced words in $\langle g_1,g_2 \rangle$,
and $g_{a_1},g_{b_1},g_{a_2},g_{b_2}$ in $\{g_1,g_2,g_1^{-1},g_2^{-1}\}$.
If either $w_1,w_2 \in W$ or $w_1,w_2 \in W^{-1}$, then
$g_{a_2} \neq g_{b_1}^{-1}$ and so
$$w_1 w_2 =
g_{a_1} s_1 g_{b_1} g_{a_1} s_1 g_{b_1} g_{a_2} s_2 g_{b_2}
g_{a_2} s_2 g_{b_2}$$
in $\langle g_1,g_2 \rangle$-reduced form.
If either $w_1 \in W,w_2 \in W^{-1}$ or $w_1 \in W^{-1},w_2 \in W$, 
then, since $s_1 \neq s_2^{-1}$
and the reduced-length of both $s_1,s_2$ is $\ell-2$,
$$w_1 w_2 =
g_{a_1} s_1 g_{b_1} g_{a_1} s' g_{b_2}
g_{a_2} s_2 g_{b_2}$$
in $\langle g_1,g_2 \rangle$-reduced form,
with $s'$ non-trivial.

Any non-trivial $\langle  W \rangle$-reduced word
is not the identity of $\langle g_1,g_2 \rangle$:
For $w = g_{a} s z s g_{b}$ in $W \cup W^{-1}$,
where $z$ is a product of two elements of $\{g_1,g_2,g_1^{-1},g_2^{-1}\}$,
call $z$ the {\em center} of $w$.
The above implies that if $w_1 \neq w_2^{-1}$
then the centers of both $w_1,w_2$ are not reduced in the $\langle g_1,g_2 \rangle$-reduced 
form of $w_1 w_2$.

Hence, if $w = w_1 w_2 \cdots w_k$ is a non-trivial $\langle W \rangle$-reduced word,
then even in its $\langle g_1,g_2 \rangle$-reduced form $w$ is not the identity
(as all centers are not reduced).
\end{proof}

Observe that for every $w \in W$,
$$\norm{w}_2, \norm{w^{-1}}_2 \leq (1+1/q)^{2 r \ell} := N.$$
Cover the ball $B_N(1)$ in $\SL_2(\R)$ with balls of radius $\eps / N$.
There exists $w_0 \in W$ so that
$$\left| B_{\eps/N}(w_0) \cap W \right| \gtrsim |W| (\eps / N^2)^3
\gtrsim \eps^3 3^{\ell} (1+1/q)^{-12 r \ell}.$$
Define
$$\g = \big( w_0^{-1} \big(  B_{\eps/N}(w_0) \cap W \big) \big) \setminus \{1\}.$$
Choose $q$ as a universal constant so that $(1+1/q)^{12r} < 1.01$.
Hence,
$$|\g| = |W| - 1 \gtrsim 2^\ell .$$
In addition, for $g \in \g$,
$$\norm{1-g}_2 \leq N\norm{w_0 - w_0 g}_2
\leq \eps,$$
and the entries of $g$ are of the form $\Z/Q$ with
$Q=q^{4r\ell}$ and $\log Q \sim \log(1/\eps)$.
Finally, as $\g$ is of the form $w_0^{-1} W \setminus \{1\}$ with $W$ freely generating a group,
the elements of $\g$ freely generate a group as well.
\end{proof}

\section{Restricted spectral gap via flattening}
\label{sec: proof of rest spec gap}

To prove the ``restricted spectral gap'' property,
we prove the following theorem that roughly states that after enough iterations
$\nu$ becomes very flat.  
Denote by $P_\delta$
the {\em approximate identity} on $\SL_2(\R)$, namely,
the density of the uniform distribution on the ball of radius $\delta$
around $1$ in $\SL_2(\R)$,
$$P_\delta = \frac{{\bf 1}_{B_\delta(1)}}{|B_\delta(1)|}.$$
For two distributions $\mu,\mu'$ on $\SL_2(\R)$
denote by $\mu * \mu'$ the convolution of $\mu$ and $\mu'$.
Denote by $\mu^{(\ell)}$ the $\ell$-fold convolution of $\mu$ with itself.

\begin{thm}
\label{thm: mu is flat}
Let $\gamma > 0$.  Assume that $\eps > 0$, the parameter from \ref{itm: g close to 1}
in Lemma~\ref{lem: free gp}, and $\delta > 0$ are small enough
as a function of $\gamma$. If 
$$\ell > C_1\frac{\log (1/\delta)}{ \log (1/\eps) } $$
with $C_1 = C_1(\gamma) > 0$, then
$$ \norm{\nu^{(\ell)} * P_\delta}_\infty < \delta^{-\gamma} .$$
\end{thm}

The proof of the theorem is given in Section~\ref{sec: flat}.
(When applying the theorem, $\gamma$ is a universal constant.)

\begin{proof}[Proof of Theorem~\ref{thm: rest spec}]
Let $f \in \f_K$.
Assume that \eqref{eqn: Tf f is small} does not hold, i.e.,
\begin{align}
\label{eqn: Tf f is small 2}
\ip{\sum_{g} \nu(g) T_g f}{f} \geq 1/2 .
\end{align}

We start by finding a level set of the Fourier transform that ``violates
\eqref{eqn: Tf f is small} as well.''
The Littlewood-Paley decomposition of $f$ is
$$f = \sum_{k =0}^\infty \Delta_k f,$$
where for every $k$ and for every $\lambda \in \supp \wh{\Delta_k f}$,
$$|\lambda| \sim 2^k.$$
We are interested in the Hecke operator
$$T = \sum_{g} \nu(g) T_g .$$

As $f \in \f_K$, 
we can consider the part of $f$ with high frequencies.

\begin{claim}
\label{clm:  K is large}
For $k_0 \geq 0$, define
$$ f_0 = \sum_{k \geq k_0} \Delta_k f.$$
If $K$ is large enough, depending on $k_0$, then
$$\ip{T f_0}{f_0} > 1/4 .$$
\end{claim}

Isolate one frequency-level of $f_0$, using the following claim.
\begin{claim}
\label{clm: exists k > k0}
There is $k \geq k_0$ so that
$$\norm{T \Delta_k f_0}_2 \geq c_1 \norm{\Delta_k f_0}_2$$
with $c_1 > 0$ a universal constant.
\end{claim}

\begin{proof}
Bound
\begin{align*}
\norm{T f_0}_2^2
& \leq \sum_{k,k'} \left| \ip{T \Delta_k f_0}{ T \Delta_{k'} f_0} \right|
= \sum_{|k-k'| \leq C} \left| \ip{T \Delta_k f_0}{ T \Delta_{k'} f_0} \right|
+ \sum_{|k-k'| > C} \left| \ip{T \Delta_k f_0}{ T \Delta_{k'} f_0} \right| 
\end{align*}
with $C > 0$ a universal constant to be determined.
Bound each of the two terms in the sum separately.
Firstly,
\begin{align*}
\sum_{|k-k'| \leq C} \left| \ip{T \Delta_k f_0}{ T \Delta_{k'} f_0} \right|
& \leq \sum_{|k-k'| \leq C} \norm{T \Delta_k f_0}_2 \norm{T \Delta_{k'} f_0}_2 
\lesssim C \sum_{k} \norm{T \Delta_k f_0}_2^2
\end{align*}
Secondly, consider $k > k' + C$.
Recall that (the absolute value of)
the spectrum of $\Delta_k f_0$ is of order $2^k$.
Similarly, the spectrum of $\Delta_{k'}f_0$ is of order $2^{k'}$,
which, since $T_g$ for $g \in (\g \cup \g^{-1})(\g \cup \g^{-1})$
is a smooth $L^{\infty}$-perturbation of identity,
implies that the norm of the derivative of $T_g \Delta_{k'} f_0$
is at most order $2^{k'}$.
Hence,
\begin{align*}
\left| \ip{T \Delta_k f_0}{ T \Delta_{k'} f_0} \right|
\lesssim 2^{-k} \norm{\Delta_k f_0}_2 2^{k'} \norm{\Delta_{k'} f_0}_2 .
\end{align*}
Thus,
\begin{align*}
\sum_{k > k' + C} \left| \ip{T \Delta_k f_0}{ T \Delta_{k'} f_0} \right| 
& \lesssim 
\sum_{k > k'+ C} 2^{k'-k} \norm{\Delta_k f_0}_2 \norm{\Delta_{k'} f_0}_2
\lesssim 2^{-C} \norm{f_0}_2^2,
\end{align*}
and so, for appropriate $C$,
\begin{align*}
\sum_{|k - k'| > C} \left| \ip{T \Delta_k f_0}{ T \Delta_{k'} f_0} \right| 
< 1/20 .\end{align*}
Concluding, using Claim~\ref{clm:  K is large},
$$
\sum_{k \geq k_0} \norm{\Delta_k f_0}_2^2
\lesssim \norm{f_0}_2^2 \lesssim
1/16 - 1/20 < 
\norm{T f_0}^2_2 - 1/20\lesssim C \sum_{k \geq k_0} \norm{T \Delta_k f_0}_2^2.$$
\end{proof}

Set
$$F = \frac{\Delta_k f_0}{\norm{\Delta_k f_0}_2}$$
with $k$ from Claim~\ref{clm: exists k > k0}.
Thus, $\ip{T F}{T F} \geq c_1^{2}$
and so $\norm{T^2 F}_2 \geq c_1^{2}$. 
Iterating, for all $\ell > 0$ a power of two,
\begin{align}
\label{eqn: tl is large}
\norm{T^\ell F}_2 \geq c_1^{\ell}.
\end{align}

To prove the theorem, argue that the norm of $T^\ell F$ is actually small,
thus obtaining the required contradiction:
Let $\gamma > 0$ be a small universal constant (to be determined).
Let $\ell$ be the smallest power of two so that
$$\ell > C_1(\gamma) k / \log (1/\eps)$$
and by Theorem~\ref{thm: mu is flat},
$$ \norm{\nu^{(\ell)} * P_\delta}_\infty < \delta^{-\gamma} ,$$
with $\eps > 0$ a small enough universal constant to be determined, 
and
$$\delta = 4^{-k}.$$

As $\delta$ is small
and the spectrum of $F$ is controlled, the following claim holds.

\begin{claim}
\label{clm: norm is large}
$$\norm{\int_{\SL_2(\R)} (T_g F) ((\nu^{(\ell)} * P_\delta)(g)) dg }_2
\gtrsim c_1^{\ell}.$$
\end{claim}

\begin{proof}
If $g = \mat{a}{b}{c}{d}$ satisfies $\norm{g-1}_2 \leq \eta \leq 1/20$,
then for all $x \in \R$ so that $|x| \leq 2$,
$$|x - g x| = \left| \frac{cx^2 + dx - ax - b}{cx + d} \right|
\lesssim \eta .$$
In addition, if $h \in B_\delta(g)$ for $g \in \supp(\nu^{(\ell)})$,
then 
$$\norm{h^{-1} g - 1}_2 \leq \delta (1+\eps)^\ell.$$
Recall, $2^k \delta (1+\eps)^\ell$ is much smaller than $c_1^{\ell}$.
Hence, since the norm of the derivative of $F$ is at most order $2^k$,
$$\norm{T_g F - T_h F}_2
= \norm{F - T_{h^{-1} g} F}_2 
\lesssim 2^k \delta (1+\eps)^\ell.$$
So,
\begin{align*}
& \norm{T^\ell F - \int_{\SL_2(\R)} (T_h F) ((\nu^{(\ell)} * P_\delta)(h)) dh}_2 
\lesssim 2^k (1+\eps)^\ell \delta \leq c_1^{\ell}/2. 
\end{align*}
The claim follows by \eqref{eqn: tl is large}.
\end{proof}

The claim above contradicts the following proposition,
as shown below.
In short, the proposition follows by the flatness lemma
and the subgroup structure of $\SL_2(\R)$.

\begin{prop}
\label{prop: norm is small}
There exists universal constants $\sigma_0,C > 0$ so that
\begin{align*}
\norm{\int_{\SL_2(\R)} (T_g F) ((\nu^{(\ell)} * P_\delta)(g)) dg }_2
\lesssim \delta^{-\gamma} (1+\eps)^{C\ell} 2^{- \sigma_0 k} .
\end{align*}
\end{prop}

\begin{proof}
Bound, using Theorem~\ref{thm: mu is flat}
and unitarity of $T_h$,
since the support of $\nu^{(\ell)} * P_\delta$
is contained in $B_{2(1+\eps)^\ell}(1)$,
\begin{align}
\label{eqn: b1} 
\nonumber \norm{\int (T_g F) ((\nu^{(\ell)} * P_\delta)(g)) dg }_2^2 &
= \int \int  \ip{T_g F}{T_h F} 
((\nu^{(\ell)} * P_\delta)(g))  ((\nu^{(\ell)} * P_\delta)(h)) dg dh 
\\  & \lesssim \delta^{-2 \gamma} (1+\eps)^{3 \ell}
\int_{B_{4(1+\eps)^{2\ell}}(1)}  \left| \ip{T_g F}{F} \right|  dg . 
\end{align}
Approximate $B_{4(1+\eps)^{2\ell}}(1)$ by a smooth function:  
let $\kappa : \SL_2(\R) \to \R_{\geq 0}$
be a smooth function so that $\norm{\kappa}_\infty = 1$, and so that
$\kappa(g) = 1$ if $\norm{g -1}_2 \leq 4(1+\eps)^{2\ell}$ 
and $\kappa(g) = 0$ if 
$\norm{g}_2 > 8 (1+\eps)^{2\ell}$.
Using Cauchy-Schwarz inequality, 
\begin{align}
\label{eqn: b3}
|\eqref{eqn: b1}| \lesssim \delta^{-2 \gamma} (1+\eps)^{5 \ell}
\left( \int  \left| \ip{T_g F}{F} \right|^2 \kappa(g)  dg \right)^{1/2}.
\end{align}
Write
\begin{align*}
\int  \left| \ip{T_g F}{F} \right|^2 \kappa(g)  dg
& \leq \int \int |F(x)| |F(y)| 
\left| \int T_g F(x) T_g F(y) \kappa(g) dg \right| dx dy .
\end{align*}
Separate to two cases, according to the distance between $x$ and $y$.
Choose $\eta > 0$ small, to be determined.
In both cases, use the following (convenient) parameterization of $\SL_2(\R)$:
$$g = \mat{a}{b}{c}{d} = 
\mat{u \cos \theta}{v \cos \phi}{u \sin \theta}{v \sin \phi}$$
with
$$uv \sin (\phi - \theta) = 1.$$
On the chart $a \neq 0$, we have
$$dg = \frac{da db dc}{|a|} = \frac{du d \theta d \phi}{|u|\sin^2(\theta - \phi)} .$$

{\bf Case one.}
The first case is when $x,y$ are close:
Bound
\begin{align}
\label{eqn: x,y close}
\int \int_{|x-y| < \eta} |F(x)| |F(y)| 
 \int |T_g F(x) |  |T_g F(y)| \kappa(g) dg  dx dy .
\end{align}
Write $F = F_1 + F_\infty$ with
$$\norm{F_1}_1 \leq 2^{-\sigma k} \ \ \text{and} \ \ 
\norm{F_\infty}_\infty \leq 2^{\sigma k}$$
for a universal constant $\sigma > 0$ to be determined.
Equation $\eqref{eqn: x,y close}$ can be bounded from above
by a sum of several terms (with different combinations of $F_1,F_\infty$ replacing $F$).
Consider, e.g., substituting $F_1$ instead of the leftmost $F$ in \eqref{eqn: x,y close},
\begin{align}
\label{eqn: x,y close one 1}
\int \int_{|x-y| < \eta}  |F_1(x)| |F(y)| 
 \int |T_g F(x)| |T_g F(y)| \kappa(g) dg  & dx dy 
\leq \int  |F_1(x)|  
 \int | T_g F(x)|  \kappa(g) dg  dx .
\end{align} 
Fix $x$, and denote
$$M= (x+1)^{-1/2} \mat{1}{-x}{1}{1} \in \SL_2(\R),$$
so that $\ol{M}(x) = 0$.
(The matrix $M$ shows two-transitivity
of the M\"{o}bius action:
$M$ maps $x$ to zero and $-1$ to infinity.
Note that $x,-1$ are far.)
Change variables and use parametrization given above,
\begin{align}
\label{eqn: x,y close fix x}
\int | T_g F(x)|  \kappa(g) dg 
& =  \int | T_{M^{-1} g^{-1}} F(x)|  \kappa(M^{-1} g^{-1}) dg \\
\nonumber & 
\lesssim \int \int \int 
 | F(\cot \phi)|  \kappa(M^{-1} g^{-1}) 
\frac{1}{|\sin \phi||\sin(\theta - \phi)|}
du d \theta d \phi .
\end{align}
If $\kappa(M^{-1} g^{-1}) \neq 0$, then $\norm{g}_2 \lesssim (1+\eps)^{2 \ell}$,
and so in the integral above $|\sin(\theta - \phi)| \gtrsim (1+\eps)^{-4\ell}$.
Change variables again, 
\begin{align*}
|\eqref{eqn: x,y close fix x}|
\lesssim (1+\eps)^{4\ell} \int \int \int 
 | F(\xi)|  \kappa(M^{-1} g^{-1}) 
\frac{1}{|\xi + 1|^{1/2}}
du d \theta d \xi 
\lesssim (1+\eps)^{6 \ell}.
\end{align*}
Hence,
\begin{align*}
|\eqref{eqn: x,y close one 1}|
\lesssim (1+\eps)^{6 \ell} \norm{F_1}_1
\leq (1+\eps)^{6\ell} 2^{-\sigma k} .
\end{align*}
The same bound holds also 
if we replace each of the other three $F$'s
by $F_1$ in \eqref{eqn: x,y close}.
It thus remains to trivially bound
\begin{align*}
\int \int_{|x-y| < \eta} |F_\infty(x)| |F_\infty(y)| &  
\int |T_g F_\infty(x)| |T_g F_\infty(y)| \kappa(g) dg dx dy
 \lesssim \eta (1+\eps)^{6 \ell} 2^{4 \sigma k} ,
\end{align*}
and conclude
\begin{align}
\label{eqn: x,y close final}
|\eqref{eqn: x,y close}|
\lesssim  (1+\eps)^{6 \ell} \left( \eta 2^{4 \sigma k} +  2^{-\sigma k} \right).
\end{align}

{\bf Case two.}
Next, understand what happens for far $x$ and $y$.
The argument in this case is more elaborate and uses knowledge of the spectrum of $F$.
Start by
\begin{align}
\label{eqn: x,y far a}
\int \int_{|x- y | \geq \eta} |F(x)| |F(y)| 
& \left| \int T_g F(x) T_g F(y) \kappa(g) dg \right| dx dy 
\\ \nonumber & \leq 
\left( \int \int_{|x- y | \geq \eta} 
\left| \int_{\SL_2(\R)} T_g F(x) T_g F(y) \kappa(g) dg \right|^2 dx dy \right)^{1/2}.  
\end{align}
In this case, argue for fixed $x$ and $y$ in $[0,1]$ so that $x \geq y + \eta$.
Denote
$$M= (x-y)^{-1/2} \mat{1}{-x}{1}{-y} \in \SL_2(\R),$$
so that $\ol{M}(x) = 0$ and $\ol{M}(y) = \infty$.
Change variables,
\begin{align*}
\left| \int T_g F(x) T_g F(y) \kappa(g) dg \right|
& =  \left| \int T_{M^{-1} g^{-1}} F(x) T_{M^{-1} g^{-1}} F(y) \kappa(M^{-1}g^{-1}) dg \right| \\
& =  (x-y)^{-1} \left|
\int \frac{F(\cot \phi) F(\cot \theta)}{| \sin \phi \cdot
 \sin \theta|} \kappa(M^{-1}g^{-1}) \frac{du d \theta d \phi}{|u||\sin(\theta - \phi)|} \right| .
\end{align*}
Change variables,
\begin{align*}
\int \frac{F(\cot \phi) F(\cot \theta)}{| \sin \phi \cdot
 \sin \theta|}  \kappa(M^{-1}g^{-1}) \frac{du d \theta d \phi}{|u||\sin(\theta - \phi)|} 
&  = \int \int  F(\xi) F(\zeta) E(\xi,\zeta)  d \xi d \zeta ,
\end{align*}
with
$$E(\xi,\zeta)=  \frac{\sqrt{(1+\xi^2)(1+\zeta^2)}}{|\sin( \cot^{-1} \zeta - \cot^{-1} \xi)|}
 \int \kappa(M^{-1}g^{-1}) \frac{du}{|u|} .$$
Continue by using that Fourier basis diagonalize $\nabla$.
Start by bounding the norms of $E$ and $\nabla E$.
First, if $\kappa(M^{-1} g^{-1}) \neq 0$, then
$$\norm{g}_2 \lesssim (1+\eps)^{2 \ell} \eta^{-1/2}.$$
Hence, in the definition of $E$ we can assume
$$(1+\eps)^{-2\ell} \eta^{1/2} \lesssim |u| \lesssim (1+\eps)^{2 \ell} \eta^{-1/2},$$
and 
$$\frac{1}{|\sin( \cot^{-1} \zeta - \cot^{-1} \xi)|} \gtrsim (1+\eps)^{-4\ell} \eta.$$
Therefore, there is a universal constant $C > 0$ so that
$$\norm{E}_\infty, \norm{ \norm{\nabla E}_2}_\infty \lesssim
 (1+\eps)^{C \ell} \eta^{-C}.$$
Since the support of the Fourier transform of $F$
is of absolute value at least order $2^k$, bound
\begin{align*}
\left| \int \int  F(z) F(w) E(z,w) dz dw \right|
& \lesssim 2^{-k} (1+\eps)^{C \ell} \eta^{-C} .
\end{align*}
Thus,
\begin{align}
\label{eqn: x,y far b}
|\eqref{eqn: x,y far a}| \leq 2^{-k} (1+\eps)^{C \ell} \eta^{-C-1} .
\end{align}

{\bf Concluding.}
By \eqref{eqn: x,y close final} and \eqref{eqn: x,y far b},
$$ \sqrt{|\eqref{eqn: b1}|}
\lesssim \delta^{-\gamma} (1+\eps)^{C \ell}
\left( \eta 2^{4 \sigma k} + 2^{-\sigma k}
+ 2^{-k}\eta^{-C} \right)^{1/2}  
\leq \delta^{-\gamma} (1+\eps)^{C \ell} 2^{-\sigma k /4} $$
for appropriate choice of $\eta$,
and with $\sigma > 0$ a universal constant.
\end{proof}

We can finally conclude, using Claim~\ref{clm: norm is large}
and Proposition~\ref{prop: norm is small},
\begin{align}
c_1^{\ell} \lesssim
|\eqref{eqn: b3}|
\lesssim \delta^{-\gamma} (1+\eps)^{C\ell} 2^{- \sigma_0 k} ,
\end{align}
which is a contradiction for $\gamma = \sigma_0/4$,
$k_0$ large and $\eps$ small.
\end{proof}

\section{Flatness via a product theorem}
\label{sec: flat}

Theorem~\ref{thm: mu is flat} follows from the following flattening lemma,
which roughly states that if 
$$\mu = \nu^{(\ell_0)} * P_\delta$$
is a little flat then $\mu * \mu$ is much flatter
(unless $\mu$ is already very flat).
The proof of the lemma is given in Section~\ref{subsec: proof of flat}.

\begin{lemma}
\label{lem: L2 flat}
Let $0 < \gamma < 3/2$.
With the notation above, 
assume that
$$\delta^{-\gamma} < \norm{\mu}_2 < \delta^{-3/2 + \gamma}$$
and 
$$\ell_0 > C_2 \frac{\log (1/\delta)}{\log (1/\eps)} $$
with $C_2 = C_2(\gamma) > 0$.
Also assume that $\eps > 0$, the parameter from \ref{itm: g close to 1}
in Lemma~\ref{lem: free gp}, and $\delta > 0$ are small enough as a function of $\gamma$.
Then,
there exists $\sigma = \sigma(\gamma) > 0$ so that
$$\norm{\mu * \mu}_2 < \delta^\sigma \norm{\mu}_2 .$$
\end{lemma}

We apply the flattening lemma iteratively.
To start iterating, we need to show that 
$\mu$ is ``a little flat'' to begin with.

\begin{prop}
\label{prop: start prop}
If
$$ \ell_0 \geq \log_Q(1/\delta)$$
with $Q$ from Lemma~\ref{lem: free gp},
then
$$\norm{\mu}_2 \leq \delta^{-3/2+\gamma}$$
with 
$\gamma > 0$ a universal constant.
\end{prop}

This follows from Kesten's bound, the following proposition
about random walks on free groups.

\begin{prop}
\label{prop: kesten}
Assume $H$ is a finite set freely generating a group.  
Denote 
$$\pi = (2|H|)^{-1} \sum_{h \in H} {\bf 1}_h + {\bf 1}_{h^{-1}}.$$
Denote by $p^{(t)}(x,x)$ the probability of being
at $x$ after $t$ steps in a random walk according to $\pi$ started at $x$.
Then,
$$\limsup_{t \to \infty}
(p^{(t)}(x,x))^{1/t} = \frac{\sqrt{2k-1}}{k} .$$
\end{prop}

Denote by $\w_k(\g)$ the set of words of length at most $k$ in $\g \cup \g^{-1}$.

\begin{proof}[Proof of Proposition~\ref{prop: start prop}]
Let $k$ be the maximal integer so that
$$1/Q^k \geq \delta^{1/2}.$$
For every $y \in \supp(\nu^{(k)})$,
$$\norm{y}_2 \leq (1+\eps)^k \leq \delta^{\eps},$$
for $\eps$ small.
By Lemma~\ref{lem: free gp},
the entries of elements in $\w_k(\g)$ are in $\Z/Q^k$.
So, for all $y \neq y'$ in $\w_k(\g)$,
$$\norm{y-y'}_2 \geq \delta^{1/2},$$
which implies
$$(y B_\delta(1)) \cap (y' B_\delta(1)) = \emptyset,$$
for $\eps$ small.
Hence,
\begin{align*}
\norm{ \sum_{y} \nu^{(k)}(y) P_\delta(y^{-1} \cdot) }_2 
\leq \left( \sum_y (\nu^{(k)} (y) )^2 \norm{P_\delta (y^{-1} \cdot)}_2^2 \right)^{1/2} 
 \leq \norm{\nu^{(k)}}_\infty^{1/2} \norm{P_\delta}_2 .
\end{align*}
Finally, by Propositions~\ref{prop: kesten}
and Lemma~\ref{lem: free gp}, 
since convolution does not increase norms,
\[ \norm{\mu}_2 \lesssim \left( \frac{ 2 |\g| - 1}{|\g|^2} \right)^{k/4}
\delta^{-3/2} < \delta^{-3/2 + \gamma}. \]
\end{proof}

\begin{proof}[Proof of Theorem~\ref{thm: mu is flat}]
By Proposition~\ref{prop: start prop}, 
and Lemmas~\ref{lem: L2 flat} and~\ref{lem: free gp},
\begin{align}
\label{eqn: mu k1}
\norm{\mu^{(k)}}_2
= \norm{ (\nu^{(\ell_0)} * P_\delta)^{(k)}}_2 \leq \delta^{-\gamma/4} 
\end{align}
with $k = k(\gamma) > 1$ and
$$\ell_0 \leq C_3 \frac{\log (1/\delta)}{\log (1/\eps)} ,$$
with $C_3  > 0$ a constant.
For every $g$,
$$\left| \mu^{(2k)} (g) \right| 
= \left| \int_{h} \mu^{(k)}(h) \mu^{(k)}(h^{-1} g) dh \right|
\leq \norm{\mu^{(k)}}_2^2 
\leq
\delta^{-\gamma/2}.$$
Lemma 2.5 in \cite{BG1} states
$$c P_\delta \leq P_\delta * P_\delta \leq \frac{1}{c} P_{2\delta}$$
with $c > 0$ a constant.
Hence,
$$\norm{\nu^{(\ell)} * P_\delta}_\infty \leq
C_4 (1+\eps)^{C_4 \ell_0} \norm{\mu^{(2k)}}_\infty 
\leq C_4 (1+\eps)^{C_4 \ell_0} \delta^{-\gamma/2} \leq \delta^{-\gamma} $$
with $C_4 = C_4(\gamma) > 0$ and 
$\ell \leq C_4 \ell_0$, for $\eps,\delta$ small.
\end{proof}

\subsection{A product theorem}
\label{subsec: proof of flat}

The flattening lemma follows from the following product theorem.
(The proof of the product theorem is deferred to Section~\ref{sec: prod thm proof}.)
We need to use {\em metric entropy}:
for a subset $S$ of a metric space denote by
$\n_\delta(S)$ the least number of balls of radius $\delta$
needed to cover $S$.

\begin{thm}
\label{thm: prod thm}
For all $\sigma_1,\tau  > 0$, there is $\eps_5 > 0$ so that the following holds.
Let $\delta > 0$ be small enough.
Let $A \subset \SL_2(\R) \cap B_\alpha(1)$,
$\alpha > 0$ a small universal constant, 
be so that 
\begin{enumerate}
\item $A = A^{-1}$, 
\item  
$$\n_\delta(A) = \delta^{-3 + \sigma_0},$$
$\sigma_1 \leq \sigma_0 \leq 3 - \sigma_1$,
\item for every $\delta < \rho < \delta^{\eps_5}$, there is a finite set
$X \subset A$ so that $|X| \geq \rho^{-\tau}$ and for every $x \neq x'$ in $X$
we have $\norm{x - x'}_2 \geq \rho$, and
\item \label{itm: cond wrt every basis} 
w.r.t. every complex basis change diagonalizing
some matrix in $\SL_2(\R) \cap B_{1}(1)$,
there is $g \in A_{(4)}$ so that 
$|g_{1,2}g_{2,1}| \geq \delta^{\eps_5}$.
\end{enumerate}
Then,
$$\n_\delta(AAA) > \delta^{-\eps_5} \n_\delta(A).$$
\end{thm}

The condition that $A$ is contained in a small ball
is not necessary, but simplifies the statement and the proof.
The condition $A = A^{-1}$ is, of course, not necessary as well,
but simplifies notation.
Condition \ref{itm: cond wrt every basis} above implies
that $A$ is far from strict subgroups.

\begin{proof}[Proof of Lemma~\ref{lem: L2 flat}]
We prove the lemma for 
$$\ell_0 \sim C_2(\gamma) \frac{\log (1/\delta)}{\log (1/\eps)} .$$
The proof for larger $\ell_0$ follows, as convolution does not increase the norm.

Assume towards a contradiction that
$$\norm{\mu * \mu}_2 > \delta^{\sigma} \norm{\mu}_2 .$$
To prove the theorem, we shall find a set $A$ that violates the product theorem.
The set $A$ will be one of the level sets of $\mu$ in the following decomposition.
Decompose $\mu$ as
$$\mu \sim \sum_j 2^{j} \chi_j,$$
where the sum is over $O( \log (1/\delta))$ values of $j$
(recall that $\mu$ is point-wise bounded by $O(1/\delta^3)$
and we can ignore points with too small $\mu$-measure),
and where $\chi_j$ is the characteristic function of a set $A_j \subset \SL_2(\R)$ so that
\begin{align}
\label{eqn: aj is sym}
A_j = A^{-1}_j.
\end{align}
Choose $j_1 < j_2$ so that
\begin{align}
\label{eqn: choose j1 j2}
2^{j_1 + j_2} \norm{\chi_{j_1} * \chi_{j_2}}_2
\gtrsim \norm{\mu * \mu}_2 / \log^2(1/\delta)
\geq 
\delta^{0+} \norm{\mu}_2.
\end{align}
Using Young's inequality, 
bound
\begin{align*}
2^{j_1 + j_2} \norm{\chi_{j_1}}_2 \norm{\chi_{j_2}}_1
\geq \delta^{0+} \norm{\mu}_2
\geq \delta^{0+} 2^{j_2} \norm{\chi_{j_2}}_2 .
\end{align*}
So, since $2^{j_2}|A_{j_2}| \leq 1$,
\begin{align}
\label{eqn: 2j1 aj1 large}
2^{j_1/2}  |A_{j_1}|^{1/2} \geq
2^{j_1 - j_2/2}  |A_{j_1}|^{1/2} \geq
2^{j_1} |A_{j_1}|^{1/2} |A_{j_2}|^{1/2} \geq \delta^{0+} .
\end{align}
Similarly,
\begin{align*}
2^{j_1/2 - j_2/2} \geq
2^{j_1/2}  |A_{j_2}|^{1/2} \geq \delta^{0+} ,
\end{align*}
which implies
$$2^{j_1} < 2^{j_2} \leq \delta^{0-} 2^{j_1}.$$

Since $2^{j_2} |A_{j_2}| \leq 1$, 
using Young's inequality and \eqref{eqn: aj is sym}, we thus have
\begin{align*}
\delta^{0+} 2^{-2j_2} |A_{j_1}|
& \leq \ip{\chi_{j_1} * \chi_{j_2}}{\chi_{j_1} * \chi_{j_2}} 
 \leq \norm{\chi_{j_2}}_2 \norm{\chi_{j_1} * \chi_{j_1} * \chi_{j_2}}_2 \\
& \leq \norm{\chi_{j_2}}_2 \norm{\chi_{j_2}}_1 \norm{\chi_{j_1} * \chi_{j_1}}_2 
 \leq 2^{-3j_2 / 2} \norm{\chi_{j_1} * \chi_{j_1}}_2 .
\end{align*}
Hence,
\begin{align}
\label{eqn: A has high ene}
\norm{\chi_{j_1} * \chi_{j_1}}_2^2 \geq
\delta^{0 +} 2^{-j_2} |A_{j_1}|^2 \geq
\delta^{0 +} 2^{-j_1} |A_{j_1}|^2 \geq
\delta^{0 +} |A_{j_1}|^{3} .
\end{align}

Use a version of Balog-Szemeredi-Gowers theorem proved in \cite{Tao}.
Denote 
$$\K = B_{r}(1) \quad \text{with} \quad
r = \delta^{-C_3(\gamma) \eps} = \delta^{0-},$$
a compact subset of $\SL_2(\R)$, 
with $C_3(\gamma) \sim C_2(\gamma)$ to be determined.
Specifically,
if $\eps$ is small enough, then
$$A_{j_1} \subset \K.$$
The {\em multiplicative energy} of $A_{j_1}$ is
$\norm{ \chi_{j_1} * \chi_{j_1} }_2^2$.
Equation \eqref{eqn: A has high ene} implies that $A_{j_1}$ has high energy.
Theorem~5.4 (or, more precisely, its proof) in \cite{Tao}
implies that, for the appropriate $C_3(\gamma)$, 
there exists $H \subset \K$ which is an approximate group, namely,
$$H = H^{-1}$$ 
and there exists a finite set $Y \subset \K$ of size
\begin{align}
\label{eqn: y small}
|Y| \leq \delta^{0-}
\end{align}
satisfying
\begin{align}
\label{eqn: hh in yh}
H  H \subset Y  H
\end{align}
so that
\begin{align}
\label{eqn: Aj1 and H close}
\delta^{0 +} |A_{j_1}| \leq |H| \leq \delta^{0 -} |A_{j_1}| .
\end{align}
In addition, there is $y \in \K$ such that 
\begin{align}
\label{eqn: a1is big}
|A_1| \geq \delta^{0 +} |A_{j_1}| ,
\end{align}
where
$$A_1 = A_{j_1} \cap yH .$$
Finally, define
$$A = \left( (A_1^{-1} A_1) \cup  ( A_1 A_1^{-1} )\right) \cap B_\alpha(1) ,$$
for $\alpha>0$ as in Theorem~\ref{thm: prod thm}.
Hence,
\begin{align}
\label{eqn: a at least a1}
|A| \geq \delta^{0+} |A_1 | \geq 
\delta^{0 +} |A_{j_1}|.
\end{align}

We now prove that $A$ violates the product theorem.
We first show that it violates the conclusion of the product
theorem and then show that
it satisfies the assumptions of the product theorem.

Using \eqref{eqn: choose j1 j2} and Young's inequality,
$$2^{j_1 + j_2} |A_{j_2}|^{1/2} |A_{j_1}| =
2^{j_1 + j_2} \norm{\chi_{j_2}}_2 \norm{\chi_{j_1}}_1
\geq \delta^{0 +} \norm{\mu}_2 \geq
\delta^{0 +} 2^{j_2} |A_{j_2}|^{1/2} .$$
Hence, using \eqref{eqn: a1is big},
\begin{align}
\label{eqn: mu a1 is large}
\mu(y H) \geq
\mu(A_1) 
\geq \delta^{0 +} 2^{j_1} |A_{j_1}| 
\geq \delta^{0 +} .
\end{align}
On the other hand,
$$\mu(y H) 
\lesssim 
\delta^{-3} \max_{z \in \supp(\nu^{(\ell_0)})} 
\left| y H \cap B_{\delta^{1-}}(z) \right| .$$
So, there is $z_0 \in \K$ so that
$$\left| H \cap S \right| \geq \delta^{3+},$$
with 
$$S = B_{\delta^{1-}}(z_0).$$
Let $Z$ be a maximal set of points in $H$ so that
for all $z \neq z'$ in $Z$, 
$$z S \cap z' S = \emptyset .$$
Bound,
$$\delta^{0 -} |H| \geq 
| H H| \geq 
|Z| 
\left| H \cap S \right|
\geq \delta^{3 + } \n_\delta(H).$$
Hence,
\begin{align}
\label{eqn: nd H at most h}
\n_\delta(H) \leq \delta^{-3 - } |H|.
\end{align}
Finally,
$$\n_\delta(AAA) \lesssim \n_\delta(H_{(6)}) \leq 
\delta^{-3 - } |H|
\leq
\delta^{-3 - } |A|
\leq \delta^{ 0 - } \n_\delta(A) .$$
So, indeed, the conclusion of the product theorem does not hold.
It remains to prove that $A$ satisfies the assumptions of the product theorem.

First, 
$$A = A^{-1}.$$

The second thing we show is that $A$ is not too small or too large.
Equation~\eqref{eqn: choose j1 j2} implies
\begin{align*}
\delta^{0+} \norm{\mu}_2 \leq
2^{j_1 + j_2} \norm{\chi_{j_1} * \chi_{j_2}}_2
\leq 2^{j_1} \norm{\chi_{j_1}}_2 2^{j_2} \norm{\chi_{j_2}}_1
\leq 2^{j_1} |A_{j_1}|^{1/2},
\end{align*}
which implies
$$\delta^{-\gamma+} \leq 2^{j_1} |A_{j_1}|^{1/2} 
 \lesssim \norm{\mu}_2 \leq \delta^{-3/2+\gamma}.$$
Thus,
$$\delta^{-2\gamma+} |A_{j_1}| 
\leq (2^{j_1} |A_{j_1}|)^2 \leq 1 $$
and, using \eqref{eqn: 2j1 aj1 large},
$$ \delta^{0 +} \leq (2^{j_1} |A_{j_1}|)^2
\lesssim \delta^{-3+2\gamma} |A_{j_1}|.$$
Therefore,
$$ \delta^{3-2\gamma+}
\leq |A_{j_1}| \leq \delta^{2\gamma-} ,$$
which implies, using \eqref{eqn: Aj1 and H close},
$$ \delta^{3-2\gamma+}
\leq |H| \leq \delta^{2\gamma-} .$$
Therefore, using \eqref{eqn: a at least a1}
and \eqref{eqn: y small},
\eqref{eqn: hh in yh}, \eqref{eqn: nd H at most h},
$$ \delta^{-2\gamma+}
\leq \delta^{-3+} |A_{j_1}|\leq
\delta^{-3+} |A|
\leq \n_\delta(A) \leq \delta^{-3-} |H| \leq \delta^{-3+2\gamma-} ,$$
or
$$\n_\delta(A) = \delta^{-3 + \sigma_0},$$
with $\sigma_1 < \sigma_0 < 3- \sigma_1$
and $\sigma_1 = 2\gamma -$.

Thirdly, we prove that $A$ is well-distributed:  
Let $\eps_5 = \eps_5(\sigma_1,\tau) >0$ be as given by Theorem~\ref{thm: prod thm}
for $\tau > 0$ a universal constant to be determined,
and let $\delta < \rho < \delta^{\eps_5}$.
We prove that there is a finite set
$X \subset A$ so that $|X| \geq \rho^{-\tau}$ and for every $x \neq x'$ in $X$
we have $\norm{x - x'}_2 \geq \rho$.
Equation \eqref{eqn: mu a1 is large}
says $\mu(A_1) \geq \delta^{0 +}$. 
Write
$\nu^{(\ell_0)} = \nu^{(\ell)} * \nu^{(\ell_0-\ell)}$,
for $\ell < \ell_0$ the largest integer so that
$$Q^{-\ell} >  \rho.$$
There thus exists $z_1 \in \K$ so that
$$\nu^{(\ell)}(A_1 z_1) \geq \delta^{0+}.$$
By Lemma~\ref{lem: free gp}, 
for every $x \neq x'$ in $\supp(\nu^{(\ell)}) \subseteq \w_{\ell}(\g)$,
$$\norm{x-x'}_2 \geq Q^{-\ell} > \rho.$$
By Proposition~\ref{prop: kesten},
$$\nu^{(\ell)} (A_1 z_1)
\leq |\w_{\ell}(\g) \cap A_1 z_1| \left( \frac{2 |\g| - 1}{|\g|^2} \right)^{\ell/2}.$$
Thus, using Lemma~\ref{lem: free gp} again,
$$\n_\rho(A)
\geq \delta^{0 +} \n_\rho(A_1 z_1) \geq 
\delta^{0 +} |\w_{\ell}(\g) \cap A_1 z_1| 
\geq 
\delta^{0+} \left( \frac{|\g|^2}{2 |\g|-1} \right)^{\ell/2} \geq \rho^{-\tau} ,$$
for $\tau \sim 1$.

It remains to show that $A$ contains matrices with certain properties.
That is, w.r.t. every basis in a bounded domain, 
there is $g \in A_{(4)}$ so that 
$|g_{1,2}g_{2,1}| \geq \delta^{\eps_5}$.
Fix a basis diagonalizing some matrix in $\SL_2(\R) \cap B_1(1)$. 
Choose $\ell_1$ large, to be determined.
By Proposition 8 from \cite{BG2},
since the elements of $\g$ freely generate a group,
if $S \subset \w_{\ell_1}(\g)$ is so that for all
$g_1,g_2,g_3,g_4 \in S$, the bi-commutator
$[[g_1,g_2],[g_3,g_4]]$ is $1$, then $|S| \leq \ell_1^6$.
As above, there is $z_2 \in \K$ so that
$$|\w_{\ell_1}(\g) \cap A_1 z_2| 
\geq 
\delta^{0+} \left( \frac{|\g|^2}{2|\g| - 1} \right)^{\ell_1/2}. $$
The set $A_1 z_2$ is contained in a ball of radius $r' = \delta^{0-}$ around $1$.
Cover the ball of radius $r'$ around $1$ by balls of radius 
$\beta = \alpha/(r'+1) \geq \delta^{0+}$.
There thus exists $z_3 \in \w_{\ell_1}(\g) \cap A_1 z_2$ so that
$$
|\w_{\ell_1}(\g) \cap A_1 z_2 \cap B_\beta(z_3)| 
\geq 
\delta^{0+} \left( \frac{|\g|^2}{2|\g|-1} \right)^{\ell_1/2} > \ell_1^6$$
(the last inequality is the first property $\ell_1$ should satisfy).
Hence,
there are 
$$g_1,g_2,g_3,g_4 \in  ( \w_{\ell_1}(\g) \cap A_1 z_2 \cap B_\beta(z_3) ) z_3^{-1}
\subset A_1 A_1^{-1}$$ 
with non-trivial bi-commutator.
For every $g' \in \{g_1,g_2,g_3,g_4\}$,
$$\norm{g' - 1}_2 \leq
\norm{g' z_3 - z_3}_2 (r'+1) \leq \beta (r'+1) = \alpha,$$
which implies
$$g' \in A.$$
If $g' \in \{g_1,g_2,g_3,g_4\}$ is so that
$|(g')_{1,2}(g')_{2,1}| \neq 0$, then 
$$|(g')_{1,2}(g')_{2,1}| \geq Q^{-20\ell_1} \geq \delta^{\eps_5}$$
(this is the second property $\ell_1$ should satisfy).
In this case, we are done.
Otherwise, recall that if four $2 \times 2$
matrices are either all upper triangular or all lower triangular, then
they have a trivial bi-commutator.
So, w.l.o.g. $g_1$ is lower triangular and $g_2$ is upper triangular,
which implies that $g_1 g_2$ has the required property.
\end{proof}

\section{A product theorem}
\label{sec: prod thm proof}

In this section we prove the product theorem, Theorem~\ref{thm: prod thm}.
The proof consists of several parts given in the following sub-sections.
(The outline of the proof follows \cite{BG1},
but the proof in our case is more elaborate.)
The theorem is finally proved in Section~\ref{subsec: proof of prod thm}.
We start this section with a brief outline of the proof of the product theorem.
We note that not only field properties are used but also 
metric properties, the argument is a multi-scale one.
Here are the steps of the proof (ignoring many technicalities).

We wish to prove that a set $A$ with certain properties becomes larger
when multiplied by itself.

(i) Assume toward a contradiction that $A_{(3)}$ is not larger than $A$.

(ii) Assuming (i), find a set $V$ of commuting matrices
which is not too small and is close to $A_{(2)}$.
To do so, use a version of the Balog-Szemeredi-Gowers theorem.

(iii) If $V$ is concentrated in a small ball,
then $AV$ will ``move $V$ around''
and hence $AV$ will be much bigger than $A$.
This is a contradiction, as $AV$ is close to $A_{(3)}$.

(iv) Otherwise, $V$ is not concentrated on any ball,
which means that it is well-distributed.
In this case, use the discretized ring conjecture,
which roughly states that a well-distributed set in $\R$ becomes
larger under sums and products.
To move from $\SL_2(\R)$ to $\R$, use matrix-trace,
which translates matrix-product to sums and products in the field.

In fact, the size of $V$ obtained is roughly $|A|^{1/3}$.
To get back to the ``correct'' order of magnitude,
we use that $A$ is far from strict subgroups in that it contains
a matrix $g$ so that $g_{1,2} g_{2,1}$ is far from zero
(w.r.t. any basis change).
In rough terms, this property of $A$ is used to show that the size of 
$V g V g V$ is $|V|^3 \sim |A|$.

\subsection{Finding commuting matrices}

In this sub-section we show that, under some non-degeneracy conditions,
a set of matrices induces a not-too-small set of commuting matrices.
To prove this, we also show that a set of matrices induces a not-too-small
trace-set.  We start by stating the results.  The proofs follow.

The {\em trace} of a matrix $g$ is $\Tr g = g_{1,1}+ g_{2,2}$.
Every $g$ in $\SL_2(\C)$ with $|\Tr g| \neq 2$ can be diagonalized.
(Elements $g$ in $\SL_2(\R)$ with $|\Tr g| < 2$ have complex eigenvalues,
so we must consider $\SL_2(\C)$.)
Define $\sd$ to be the set of diagonal matrices $v$ in $\SL_2(\C)$ so
that $\Tr v \in \R$.

The following lemma shows that, at least in one ``direction,''
the trace-set of a set is not too small.

\begin{lemma}
\label{lem: large traces}
Think of $\SL_2(\R)$ as a subset of $\R^4$,
and let $g_0,g_1,g_2,g_3 \in \SL_2(\R) \cap B_{1/2}(1)$ be so that
\begin{align}
\label{eqn: det gtr 1}
|\det(g_0,g_1,g_2,g_3)| \geq \delta^{0+},
\end{align}
and let $A \subset \SL_2(\R) \cap B_{1/2}(1)$.
Then, there is $I \subset \{0,1,2,3\}$ of size $|I| = 3$ so that
$$\prod_{i \in I} \n_\delta(\Tr g_i^{-1} A)
\geq \delta^{0+}\n_\delta(A). $$
\end{lemma}

The following lemma allows to find a commuting set of matrices via trace.

\begin{lemma}
\label{lem: sim dia elts}
Let $A \subset \SL_2(\C) \cap B_{\alpha}(1)$, 
$\alpha > 0$ a small constant, be so that $\dist(A, \pm 1) \geq \delta^{0+}$.
Then, there exists a set $V \subset \SL_2(\C)$ of commuting matrices so that 
$$\n_\delta(V) \geq \delta^{0+} \frac{\n_\delta(\Tr A) \n_\delta(A)}{\n_\delta(A^2 A^{-1})},$$ and every $v \in V$ satisfies $\dist(v,A^{-1} A) \leq \delta^{1-}$.
\end{lemma}

We shall also need the following corollary of the two lemmas.

\begin{cor}
\label{cor: A yield diagonal}
Let $A \subset \SL_2(\R) \cap B_\alpha(1)$,
$\alpha > 0$ a small constant.
Let $g_1,g_2,g_3 \in \SL_2(\R) \cap B_\alpha(1)$ be so that $|\det(1,g_1,g_2,g_3)| \geq \delta^{0+}$.
Then, there is a set of commuting matrices $V \subset \SL_2(\C)$ so that
there is $g_0 \in \{1,g_1,g_2,g_3\}$ so that 
$$\n_\delta(V) \geq \delta^{0+}  \frac{\n_\delta(A)^{4/3}}{\n_\delta(A g_0^{-1} A A^{-1})},$$
and every $v \in V$ satisfies $\dist(v,A^{-1} A) \leq \delta^{1-}$.
\end{cor}

\begin{proof}[Proof of Lemma~\ref{lem: large traces}]
For $i \in \{0,1,2,3\}$, denote 
$$g_i' = \mat{ d_i}{ -c_i}{ -b_i}{ a_i},$$
where
$$g_i = \mat{a_i}{b_i}{c_i}{d_i}.$$
By \eqref{eqn: det gtr 1},
\begin{align*}
|\det(g_0',g_1',g_2',g_3')| = |\det(g_0,g_1,g_2,g_3)| \geq \delta^{0+}.
\end{align*}
Hence, let $A' \subset A$ be contained in a ball of radius $\delta^{0+}$
so that 
$$\n_\delta(A) \leq \delta^{0-} \n_\delta(A'),$$
and so that there is a set $I \subset \{0,1,2,3\}$ of size $|I| = 3$ so that
$$\n_\delta(A') \leq \delta^{0-} \n_\delta(P A'),$$
where $P$ is the projection to the sub-space $\mathsf{span} \{g'_i : i \in I\}$.
(The map $g \mapsto Pg$
restricted to a small ball is a diffeomorphism with bounded distortion.)
For every $g = \mat{a}{b}{c}{d}$ in $\SL_2(\R)$,
$$\Tr g_i^{-1} g =  d_i a - b_i c - c_i b + a_i d  = \ip{g}{g_i'},$$
with the standard inner product over $\R^4$.
Thus,
$$\n_\delta(P A') 
\leq \delta^{0-} \prod_{i \in I} \n_\delta(\Tr g_i^{-1} A') \leq  
\delta^{0-} \prod_{i \in I} \n_\delta(\Tr g_i^{-1} A) .$$
\end{proof}

\begin{proof}[Proof of Lemma \ref{lem: sim dia elts}]
Choose $T \subset \Tr A$ so that
\begin{align}
 \label{eqn: T prop 1}
|T| \sim \n_\delta(\Tr A), 
\end{align}
and so that for all $t \neq t'$ in $T$,
\begin{align*} 
\label{eqn: T prop 2}
|t - t'| ,|t-2|,|t+2| > 2 \delta.
\end{align*}
(If $\n_\delta(\Tr A)$ is small, the lemma trivially holds.)
Since trace is continuous,
\[ \sum_{t \in T} \n_\delta \left( \left\{ g \in A^2 A^{-1} : |\Tr g - t| < \delta /4 \right\} \right) 
\lesssim \n_{\delta}(A^2 A^{-1}).\]
There thus exists $t_0 \in T$ so that the set
$$A_0 = \{ g \in A^2 A^{-1} : |\Tr g - t_0 | < \delta/4 \}$$
satisfies
$$\n_\delta(A_0) \lesssim \frac{\n_{\delta}(A^2 A^{-1})}{|T|} .$$
Choose $g_0 \in A$ so that $\Tr g_0 = t_0$.

Choose $A_1 \subset A_0$ so that
\[ |A_1| = \n_\delta(A_0) \]
and
\begin{align}
\label{eqn: A1 prop 2}
A_0 \subset \bigcup_{g \in A_1} B_\delta(g) .
\end{align}
For $g \in A_1$, define (with a slight abuse of notation)
$$A_g = \{ x \in A : x g_0 x^{-1} \in B_\delta(g) \}.$$
Since for every $x$ we have $\Tr x g_0 x^{-1} = \Tr g_0 = t_0$,
for every $x \in A$ we have $x g_0 x^{-1} \in A_0$.
Equation \eqref{eqn: A1 prop 2} thus implies
\[A = \bigcup_{g \in A_1} A_g .\]
Hence, there is $g_1 \in A_1$ so that
\begin{align}
\label{eqn:  of ag1}
\n_\delta(A_{g_1}) \geq \frac{\n_\delta(A)}{|A_1|} 
= \frac{\n_\delta(A)}{\n_\delta(A_0)}
\gtrsim \frac{\n_\delta(A)}{\n_{\delta}(A^2 A^{-1})} |T| . 
\end{align}

Fix $x_1 \in A_{g_1}$.  By definition,
for every $x \in A_{g_1}$,
\[ \norm{ x g_1 x^{-1} - x_1 g_1 x_1^{-1} } \leq 2 \delta . \]
Since $A$ is bounded,
\[ \norm{y g_1 - g_1 y} \lesssim \delta ,\]
where $$y = x_1^{-1} x \in x_1^{-1} A_{g_1} . $$
Since $g_1 \in A$ is far from $\pm 1$,
conclude that diagonalizing $g_1$
makes $x_1^{-1} A$ close to diagonal:
Since $|\Tr g_1| \neq 2$,
there exists a matrix $u$ so that $v_1 = u g_1 u^{-1}$ is diagonal.
By assumption on $A$,
$$\dist(v_1,\pm 1) \sim \dist(g_1,\pm 1) \geq \delta^{0+}.$$
So,
$$|(v_1)_{1,1} - (v_1)_{2,2}| \geq \delta^{0+}.$$
In addition,
$$\norm{ u y u^{-1} v_1  - 
v_1 u y u^{-1} } \lesssim \delta .$$
Hence,
$$ |(u y u^{-1})_{1,2}| , | (u y u^{-1})_{2,1}| \lesssim \delta^{1-} .$$
Since $|\det(u y u^{-1})| = 1$, there is thus a diagonal $v \in \SL_2(\C)$ so that
$$\norm{u y u^{-1} - v} \lesssim \delta^{1-}.$$
We can thus conclude that 
$x_1^{-1} A_{g_1} \subset A^{-1} A$ is in a $(\delta^{1-})$-neighborhood
of a set $V \subset \SL_2(\C)$ of commuting matrices. 
In particular,
\[ \n_\delta(V) \geq 
\delta^{0+} \n_\delta(A_{g_1}).\]
Equations \eqref{eqn:  of ag1} and \eqref{eqn: T prop 1} imply the claimed lower bound
on $\n_\delta(V)$.

\end{proof}

\begin{proof}[Proof of Corollary \ref{cor: A yield diagonal}]
Since $|\det(1,g_1,g_2,g_3)| \geq \delta^{0+}$,
the pairwise distances between $\pm 1$, $\pm g_1$, $\pm g_2$, $\pm g_3$ are at least $\delta^{0+}$.
Thus, there exists a subset $A'$ of $A$ so that
$$\n_\delta(A') \geq \delta^{0+} \n_\delta(A)$$
and
$$\dist(A',\{\pm 1,\pm g_1,\pm g_2,\pm g_3\}) \geq \delta^{0+}.$$
By Lemma~\ref{lem: large traces},
there exists $g_0 \in \{1,g_1,g_2,g_3\}$ so that
\[ \n_\delta(\Tr g_0^{-1} A') \geq \delta^{0+} \n_\delta(A')^{1/3} .\] 
Now, apply Lemma~\ref{lem: sim dia elts} on the set $g_0^{-1} A'$ to complete the proof.
\end{proof}

\subsection{Trace expansion via discretized ring conjecture}

The following lemma is the main result of this section.
The lemma roughly tells us that if a set $V$ of commuting matrices
is well-distributed then adding a non-commuting element to $V$
makes its trace-set grow under products.

\begin{lemma}
\label{lem: trace expansion for large v}
For every $0 < \sigma < 2$ and $0 < \kappa < 1$,
there is $\eps_4 > 0$ so that the following holds.
Let $V \subset \SL_2(\C) \cap B_\alpha(1)$,
$\alpha > 0$ a small constant,
be 
so that $V = V^{-1}$, 
so that $\dist(v,\sd) \leq \delta^{1-}$ for all $v$ in $V$,
so that
$$\n_\delta(V) = \delta^{-\sigma},$$
and so that for all $\delta < \rho < \delta^{\eps_4}$,
\begin{align}
\label{eqn: max a in main lem}
\max_a \n_\delta(V \cap B_\rho(a)) < \rho^\kappa \delta^{-\sigma}.
\end{align}
Let $g = \mat{a}{b}{c}{d} \in \SL_2(\C) \cap B_\alpha(1)$ be so that
$\Tr g \in \R$ and $|bc| \geq\delta^{\eps_4}$.
Then,
$$\n_\delta ( \Tr Wg Wg ) \geq \delta^{-\sigma - \eps_4},$$
where $W = V_{(8)}$.
\end{lemma}

The starting point here is the discretized ring conjecture.
This conjecture was first prove in \cite{B}
and later strengthened in \cite{BG1},
see Proposition~3.2 in \cite{BG1}.

\begin{lemma}
\label{lem: disc ring conj}
For all $0<\sigma,\kappa < 1$,
there is $\eps_2 > 0$ so that for all $\delta > 0$ small, the following holds.
Let $A \subset [-1,1]$ be
a union of $\delta$-intervals so that
$$|A| = \delta^{1-\sigma}$$
and for all $\delta < \rho < \delta^{\eps_2}$,
$$\max_a |A \cap B_\rho(a)| < \rho^\kappa |A|.$$
Then,
$$| A + A| + | A  A| > \delta^{1- \sigma -\eps_2}.$$
\end{lemma}

The discretized ring conjecture was used in \cite{BG1} to 
prove ``scalar amplification,''
i.e., the following proposition.

\begin{prop}
\label{prop: scalar amp comp}
For all $0 < \sigma, \kappa < 1$, 
there is $\eps_3 > 0$ so that the following holds.
Let $S \subset \C$ be a subset of the complex unit circle,
so that $S$ is a union of $\delta$-arcs, 
$\delta > 0$ small enough,
so that $S = S^{-1}$,
so that
\[ |S| = \delta^{1-\sigma}  \]
(size is measured in the unit circle),
and so that for all $\delta < \rho < \delta^{\eps_3}$,
\begin{align}
\label{eqn: cond on s in trace expansion comp}
\max_a |S \cap B_\rho(a)| < \rho^\kappa |S| .
\end{align} 
If $\gamma,\lambda \in \R$ are so that 
$\gamma > 0,|\lambda| \geq \delta^{\eps_3}$, then the set
\[ D = \{xy + \gamma /(xy) + 
\lambda (x/y + y/x ) : x,y \in S_{(4)} \} \]
satisfies
\[ \n_\delta(D) \geq \delta^{-\eps_3 - \sigma} .\]
\end{prop}

We also need and prove the following variant of scalar amplification.

\begin{prop}
\label{prop: scalar amp real}
For all $0 < \sigma, \kappa < 1$, 
there is $\eps_3 > 0$ so that the following holds.
Let $S \subset [1/2,2]$ be a union of $\delta$-intervals, $\delta > 0$ small enough,
so that $S = S^{-1}$, so that
\[ |S| = \delta^{1-\sigma} , \]
and so that for all $\delta < \rho < \delta^{\eps_3}$,
\begin{align}
\label{eqn: cond on s in trace expansion real}
\max_a |S \cap B_\rho(a)| < \rho^\kappa |S| .
\end{align} 
If $\gamma,\lambda \in \R$ are so that 
$\gamma > 0,|\lambda| \geq \delta^{\eps_3}$, then the set
\[ D = \{xy + \gamma /(xy) + \lambda (x/y + y/x ) : x,y \in S_{(4)} \} \]
satisfies
\[ \n_\delta(D) \geq \delta^{-\eps_3-\sigma} .\]
\end{prop}

Lemma~\ref{lem: trace expansion for large v} follows from scalar amplification.

\begin{proof}[Proof of Lemma \ref{lem: trace expansion for large v}]
Let $V_0 \subset \sd$ be so that
$\dist(v,V_0) \leq \delta_0 = \delta^{1-}$ for all $v$ in $V$
and $\dist(v_0,V) \leq \delta_0$
for all $v_0$ in $V_0$.
Specifically, 
for all $\delta_0 < \rho < {\delta_0}^{2 \eps_4}$,
\begin{align}
\label{eqn: V0 is good}
\max_a \n_{\delta_0}(V_0 \cap B_\rho(a)) \leq \delta^{0-} 
\max_a \n_{\delta_0} (V \cap B_\rho(a)) 
\leq \delta^{0-} \rho^\kappa \delta^{-\sigma}.
\end{align}
Observe
\begin{align}
\label{eqn: tr vgvg}
\Tr \mat{x}{}{}{1/x} g \mat{y}{}{}{1/y} g = a^2 xy + d^2/(xy)
+ bc (x/y + y/x).
\end{align}
Write 
$$V_0 = \left\{ \mat{x}{}{}{1/x} : x \in T\right\}.$$
The set $T$ is contained in the real numbers union the complex unit circle.
Denote by $T_1= T \cap \R$,
and $T_2 = T \setminus T_1$.
First, assume 
\begin{align}
\label{eqn: T1 is good}
\n_{\delta_0}(T_1) \sim \n_{\delta_0}(V_0).
\end{align} 
Define $S_1$ to be a $\delta_0$-neighborhood of $T_1$.
Thus,
\[ |S_1| = \delta_0^{1-\sigma_1} \]
with  $\sigma_1 \geq \sigma/2$.
Equation \eqref{eqn: V0 is good}
implies that $S_1$ satisfies \eqref{eqn: cond on s in trace expansion real}
with $\kappa_1 = \kappa/2$.
As in Propositions~\ref{prop: scalar amp real},
denote
$$D_1 = a^2 \{ x y + \gamma/ (xy) 
+ \lambda (x/y + y/x) : x,y \in 
(S_1)_{(4)}  \} .$$
with $\gamma = (d/a)^2$ and $\lambda = bc/a^2$.
Observe, $ad-bc=1$ and $a+d \in \R$ imply $d/a \in \R$
and $bc/a^2 \in \R$.
In addition,  $|\lambda| \geq \delta_0^{0+}$.
The proposition thus implies 
$$\n_{\delta_0}(D_1) \geq \delta_0^{-\eps_3 - 1} |S_1|
\geq \delta^{-\eps_3 -\sigma + }.$$
Using \eqref{eqn: tr vgvg}, conclude
\begin{align*}
\n_\delta(\Tr WgWg) \geq 
\delta^{-\sigma-\eps_3+} .
\end{align*}
When \eqref{eqn: T1 is good} does not hold,
consider $T_2$ and use Proposition~\ref{prop: scalar amp comp} instead of
Proposition~\ref{prop: scalar amp real}.
\end{proof}

\begin{proof}[Proof of Proposition~\ref{prop: scalar amp real}]
Assume towards a contradiction that the proposition does not hold.
W.l.o.g., for every $s$ in $S$,
\begin{align}
\label{eqn: s is far from a}
\dist(s,\{\gamma^{1/4} ,1\}) \geq \delta^{0+} .
\end{align}
We first find a set $A$ so that $A+A$
is not much larger than $A$.
If $s,s' \in S$, then $x = s'/s \in S_{(2)}$ and
 $y = ss' \in S_{(2)}$
satisfy $xy = {s'}^2$ and $y/x = {s}^2$.
By assumption, we can thus conclude
\begin{align*}
\left| \left\{ ({s'}^2 + \gamma/{s'}^2) + \lambda ({s}^2 + 1/{s}^2) : s',{s} \in S_{(2)} \right\} \right| 
\lesssim \delta^{-\eps_3} |S|.
\end{align*}
Denote 
$$A = \{ \lambda ({s}^2 + 1/{s}^2) : {s} \in S_{(2)} \}$$
and 
$$A' = \{ {s'}^2 + \gamma/{s'}^2 : {s'} \in S_{(2)} \} .$$
Since $| \lambda| \geq \delta^{0+}$,
$$|A| \geq \delta^{0+} |S|.$$
By \eqref{eqn: s is far from a},
the derivative of the map ${s'} \mapsto {s'}^2 + \gamma/{s'}^2$ 
is bounded away from zero in the relevant range.
Thus,
$$|A'| \geq \delta^{0+} |S|.$$
Ruzsa's inequality in measure version for open sets $A,A' \subset \R$ 
states $|A + A| \leq |A + A'|^2/|A'|$
(see, e.g., Lemma 3.2 in \cite{Tao}). 
Therefore,
\begin{align}
\label{eqn: a + a is small}
|A+A| \leq \delta^{0-} |S|.
\end{align}

We now find a set that does not significantly increase its size under sums and products.
Define 
$$A_1 = \{s^2 + 1/s^2 : s \in S  \}.$$
By \eqref{eqn: s is far from a},
$$|A_1| \geq \delta^{0+} |S| .$$
Hence, by \eqref{eqn: a + a is small}, since $|\lambda| \geq \delta^{0+}$,
$$|A_1 + A_1 | \leq \delta^{0-} |A+A| \leq \delta^{0-}|A_1|.$$
Observe
\begin{align*}
(s_1^2 + 1/s_1^2)(s_2^2 + 1/s_2^2) = 
((s_1 s_2)^2 + 1/(s_1s_2)^2) + ( (s_1/s_2)^2 + 1/(s_1/s_2)^2).
\end{align*}
Hence, using \eqref{eqn: a + a is small}, since $|\lambda| \geq \delta^{0+}$,
\begin{align*}
|A_1 A_1| & 
 \leq \delta^{0-} |A + A|  \leq \delta^{0-}|A_1|.
\end{align*}
So,
$$|A_1 + A_1| + |A_1 A_1| \leq \delta^{0-} |A_1|.$$
If $\eps_3 > 0$ is small enough, we can set $0 < \sigma' < 1$ so that
$$|A_1| = \delta^{1-\sigma'}.$$
Choose $\kappa' = \kappa/2$.
Set $\eps_2 = \eps_2(\sigma',\kappa') > 0$ as in Lemma~\ref{lem: disc ring conj}.
If $\eps_3 > 0$ is small enough, then for every $\delta < \rho < \delta^{\eps_2}$,
\begin{align*}
\max_a |A_1 \cap B_\rho(a)|
\leq \delta^{0-} \max_{a} |S \cap B_\rho(a)| < \delta^{0-} \rho^\kappa |S| \leq
\delta^{0-} \rho^\kappa |A_1| \leq \rho^{\kappa'} |A_1| .
\end{align*}
This contradicts Lemma~\ref{lem: disc ring conj}.
\end{proof}

\subsection{Expansion using a non-commuting element}

We shall use the following variant 
of a lemma from \cite{BG1}, see \cite{H} as well.
Roughly, the lemma states that adding a non-commuting element
to a commuting set of matrices makes it grow under products.

\begin{lemma}
\label{lem: V and g is larger}
Let $V \subset \SL_2(\C) \cap B_{\alpha}(1)$,
$\alpha$ a small constant,
be so that $\dist(v,\sd) \leq \delta^{1-}$ for all $v$ in $V$.
Let $g = \mat{a}{b}{c}{d} \in \sd \cap B_\alpha(1)$
be so that $|bc| \geq \delta^{0+}$.
Then,
$$\n_\delta(V g V g  V) \geq 
\delta^{0+} \n_{\delta}(V)^{3} . $$
\end{lemma}

\begin{proof}
Assume
\begin{align}
\label{eqn: V is large}
\n_{\delta}(V) > \delta^{0-}
\end{align}
(otherwise, the lemma trivially holds).
There are several cases to consider.

{\bf 1.} Denote by $\sd_\R$ the set of matrices in $\sd$ with entries in $\R$.
Consider the case that there is a subset of $\sd_\R$
with comparable metric entropy to that of $V$:
Assume that there is 
$Z \subset \R$ so that $|Z| \geq \delta^{0+} \n_\delta(V)$,
so that for all $z \in Z$,
 $$\dist \left(\mat{z}{}{}{1/z} , V \right) \leq \delta^{1-},$$
and so that for all $z \neq z'$ in $Z$,
$$|z - z'| > \delta.$$
W.l.o.g., assume that $z \geq \sqrt{d/a}$
(the proof in the other case is similar).
Furthermore, by \eqref{eqn: V is large},
we can assume w.l.o.g. that
$$z - \sqrt{d/a},|z-1| \geq \delta^{0+}.$$

For $z = (z_1,z_2,z_3)$ in $Z^3$, denote
$$M_{z} = \mat{z_1}{}{}{1/z_1} g \mat{z_2}{}{}{1/z_2} g \mat{z_3}{}{}{1/z_3} .$$
To prove the lemma, we will show that for all $z \neq z'$ in $Z^3$,
$$\norm{M_z - M_{z'}} \geq \delta^{1+} .$$

Observe
$$M_z = \mat{z_1 z_3 (a^2 z_2  + bc/z_2)}{(z_1/z_3)b(a z_2 + d/z_2)}
{(z_3/z_1) c(az_2 + d/z_2)}{(1/z_1 z_3)(bc z_2+d^2/z_2)} .$$
Consider the following two cases.

{\bf 1.1.} The first case is when $z_2 > z'_2$.
We have two sub-cases to consider.

{\bf 1.1.1.}
The first sub-case is $|z_1/z_3 - z'_1/z'_3| \geq \delta^{1+}$.
Bound
\begin{align*}
\big| (M_z)_{1,2} / (M_{z})_{2,1}  - (M_{z'})_{1,2} / (M_{z'})_{2,1} \big|
& = |b/c| \cdot \big|  (z_1/z_3)^2 - (z'_1/z'_3)^2 \big| 
\geq \delta^{1+} .
\end{align*}
Thus, 
\begin{align*}
\delta^{1+} & \leq
\big| (M_z)_{1,2} (M_{z'})_{2,1} - (M_{z'})_{1,2} (M_{z})_{2,1}  \big| \\ 
& = \big| \big( (M_z)_{1,2}  - (M_{z'})_{1,2} \big) (M_{z'})_{2,1} + 
(M_{z'})_{1,2} \big( (M_{z'})_{2,1} -  (M_{z})_{2,1} \big)  \big| .
\end{align*}
So,
$$\norm{M_z - M_{z'}} \geq \delta^{1+}.$$

{\bf 1.1.2.}
The second sub-case is $|z_1/z_3 - z'_1/z'_3| < \delta^{1+}$.
Bound
\begin{align*}
|(M_z)_{1,2} - (M_{z'})_{1,2}| 
& = |ba|  \big| (z_1/z_3)(z_2 + (d/a)/z_2) - (z'_1/z'_3)( z'_2 + (d/a)/z'_2) \big| \\
& \gtrsim |ba|  \big| z_2 + (d/a)/z_2 - z'_2 + (d/a)/z'_2 \big| - \delta^{1+} .
\end{align*}
The map $z_2 \mapsto z_2 + (d/a)/z_2$ has derivative at least $\delta^{0+}$
for $z_2 \geq \sqrt{d/a} + \delta^{0+}$.
So,
\begin{align*}
|(M_z)_{1,2} - (M_{z'})_{1,2}| \geq \delta^{1+}.
\end{align*}

{\bf 1.2.}
The second case is $z_2 = z'_2$ and $(z_1,z_3) \neq (z'_1,z'_3)$.
Assume w.l.o.g. $z_1 \neq z'_1$
(the argument in the other case is similar).
Since the entries of $g \mat{z_2}{}{}{1/z_2}g$
are bounded away from $0$ and $V$ is close to $1$,
\begin{align*}
\norm{M_z - M_{z'}} \geq \delta^{0+} \norm{ (z_1 z_3 - z'_1 z'_3,z_1z'_3 - z'_1 z_3)} .
\end{align*}
Since $\norm{(z_3,z'_3)} \gtrsim 1$ and $\left| \det \mat{z_1}{-z'_1}{-z'_1}{z_1} \right| \gtrsim \delta$,
$$\norm{ (z_1 z_3 - z'_1 z'_3,z_1z'_3 - z'_1 z_3)} \gtrsim \delta.$$

{\bf 2.}
Otherwise,
there is a subset of $\sd \setminus \sd_\R$
with comparable metric entropy to that of $V$:
There is a subset of the complex unit circle
$Z$ so that $|Z| \geq \delta^{0+} \n_\delta(V)$,
so that for all $z \in Z$,
 $$\dist \left(\mat{z}{}{}{1/z} , V \right) \leq \delta^{1-},$$
and so that for all $z \neq z'$ in $Z$,
$$|z - z'| > \delta.$$
Assume w.l.o.g. that $\dist(Z,1) \geq \delta^{0+}$.
Also assume w.l.o.g. that every element of $Z$ has positive imaginary part
(the other case is similar).

{\bf 2.1.} When $z_2 \neq z'_2$, bound
\begin{align*}
\big| |(M_z)_{1,2}| - |(M_{z'})_{1,2}| \big|
= |ba| \big| | z_2 + (d/a)/z_2| - |z'_2 + (d/a)/z'_2| \big|  .
\end{align*}
If we denote, $z_2 = e^{i\theta_2}$ and $z'_2 = e^{i \theta'_2}$, then
$$\big| | z_2 + (d/a)/z_2|^2 - |z'_2 + (d/a)/z'_2|^2 \big| =
2(d/a) \big| \cos (2 \theta_2) - \cos (2 \theta'_2) \big|
\geq \delta^{0+} |z_2 - z'_2| > \delta^{1+}.$$
Hence,
\begin{align*}
\norm{M_z - M_{z'}} \geq \delta^{1+}.
\end{align*}

{\bf 2.2.} When $z_2 = z'_2$,
the argument is similar to the one in case 1.2. above.
\end{proof}

\subsection{Finding ``independent directions''}

Roughly, we now show that two non-commuting matrices induce
four ``independent directions.''

\begin{claim}
\label{clm: finding indep}
Let $g_1 \in \SL_2(\C) \cap B_1(1)$ be so that $\dist(g_1,\pm 1) \geq \delta^{0+}$
and $\Tr g_1 \neq 2$.
Let $g_2 \in \SL_2(\C)$ be so that
w.r.t. the basis that makes $g_1$ diagonal
$|(g_2)_{1,2} (g_2)_{2,1}| \geq \delta^{0+}$.
Then,
$$|\det(1,g_1,g_2,g_1 g_2)| \geq \delta^{0+}.$$
\end{claim}

\begin{proof}
Choose a basis so that $g_1$ is diagonal
(this is a linear transformation on the $g_i$'s with bounded
away from zero determinant).
Denote $\lambda = (g_1)_{1,1}$.
In the new basis,
\begin{align*}
|\det(1,g_1,g_2,g_1 g_2)| & = 
\left| \left( \begin{array}{cccc} 
1 & \lambda & (g_2)_{1,1} & (g_1 g_2)_{1,1} \\
   &                   & (g_2)_{1,2} &  (g_1 g_2)_{1,2}\\
   &                   & (g_2)_{2,1} & (g_1 g_2)_{2,1} \\
1 & 1/\lambda & (g_2)_{2,2} &   (g_1 g_2)_{2,2}\\
\end{array} \right) \right|  \\ & = |(\lambda - 1/\lambda)
((g_1 g_2)_{1,2} (g_2)_{2,1} - (g_1 g_2)_{2,1} (g_2)_{1,2}) |.
\end{align*}
By choice,
$$|\lambda - 1/\lambda| \geq \delta^{0+}.$$
and 
$$|(g_2)_{1,2} (g_2)_{2,1}| \geq \delta^{0+}.$$
Hence,
$$|((g_1 g_2)_{1,2} (g_2)_{2,1} - (g_1 g_2)_{2,1} (g_2)_{1,2})|
= |(\lambda - 1/\lambda) (g_2)_{1,2} (g_2)_{2,1}| \geq \delta^{0+} .$$
\end{proof}

\subsection{Proof of product theorem}
\label{subsec: proof of prod thm}

\begin{proof}[Proof of Theorem~\ref{thm: prod thm}]
Assume towards a contradiction that 
$$\n_\delta(A A A) \leq \delta^{0-} \n_\delta(A).$$
By \cite{Tao}, for every finite $k$,
\begin{align}
\label{eqn: A doesnt increase}
\n_\delta(A_{(k)}) \leq \delta^{0-} \n_\delta(A)
\end{align}
as well.

The first step is to find a large, commuting set of matrices.
By assumption on $A$ and using Claim~\ref{clm: finding indep}, choose $g_1,g_2,g_3$ in $A_{(8)}$ with 
$|\det(1,g_1,g_2,g_3)| \geq \delta^{0+}$.
Equation~\eqref{eqn: A doesnt increase} and
Corollary~\ref{cor: A yield diagonal} imply that 
there is a set of commuting matrices $V \subset \SL_2(\C)$
so that 
\begin{align}
\label{eqn: V is at least A}
\n_\delta(V) \geq  \delta^{0+} \n_\delta(A)^{1/3}
= \delta^{-1+\sigma_0/3 +} 
\end{align}
and so that
$$V \subset \Gamma_{\delta^{1-}}(A_{(2)}).$$
Assume (by perhaps allowing $V \subset \Gamma_{\delta^{1-}}(A_{(4)})$) 
that $V = V^{-1}$ and
\begin{align}
\label{eqn: v in small ball}
V \subset B_{\delta^{3\eps_5}}(1).
\end{align}
Proceed according to two cases.

The first case is when $V$ is well-spread, i.e.,
the conditions for using the discretized ring conjecture are held.
Define
$$\sigma = 1-\sigma_0/3 - \quad \text{and} \quad
\kappa = \tau/6$$
so that $\n_\delta(V) = \delta^{-\sigma}$.
Assume that for all $\delta < \rho < \delta^{\eps_4}$
with $\eps_4 = \eps_4(\sigma,\kappa)$ from Lemma~\ref{lem: trace expansion for large v},
$$\max_a \n_\delta(V \cap B_\rho(a)) < \rho^\kappa \delta^{-\sigma}.$$
By assumption on $A$, there is $g_0 \in A_{(4)}$ so that
(w.r.t. the basis that makes $V$ diagonal)
the distance between $g_0$ and $1$ is at most a small constant,
and $|(g_0)_{1,2}(g_0)_{2,1}| \geq \delta^{\eps_5}$.
Even after the basis change $\Tr g_0 \in \R$.
Thus, Lemma~\ref{lem: trace expansion for large v} implies
$$\n_\delta ( \Tr W_0  ) \geq \delta^{-\sigma - \eps_4},$$
where 
$$W_0 = W g_0 W g_0 W$$
and
$$W = V_{(8)} . $$ 
(Here and below $C > 0$ will be a large universal constant,
that may change its value.)
By choice,
$$\dist(g_0^2,\pm 1) \gtrsim \delta^{\eps_5}.$$
Thus, using \eqref{eqn: v in small ball},
$$\dist(W_0,\pm 1) \gtrsim \delta^{2\eps_5}.$$
We can hence apply Lemma~\ref{lem: sim dia elts} 
with $W_0$ to obtain a set 
$$W_1 \subset \Gamma_{\delta^{1-}}(W_0^{-1} W_0)$$ 
of commuting matrices so that 
$$\n_\delta(W_1) \geq \delta^{0+} \frac{\n_\delta(\Tr W_0) \n_\delta(W_0)}{
\n_\delta(W_0^2 W_0^{-1})}
\geq \delta^{0+} \frac{\delta^{-\sigma-\eps_4} \n_\delta(V g_0
V g_0 V)}{
\n_\delta(W_0^2 W_0^{-1})}.$$
By \eqref{eqn: A doesnt increase} and Lemma~\ref{lem: V and g is larger},
we thus have
$$\n_\delta(W_1) \geq 
\delta^{0+} \frac{\delta^{-\sigma-\eps_4} \n_\delta(V)^3}{
\n_\delta(A)}.$$
So, by \eqref{eqn: V is at least A},
$$\n_\delta(W_1) \geq \delta^{-\sigma-\eps_4/2} .$$
Again, we can find $g_1 \in A_{(4)}$ so that
(w.r.t. the basis that makes $W_1$ diagonal)
$\dist(g_1,1)$ is at most a small constant, $\Tr g_1 \in \R$, and
$|(g_1)_{1,2}(g_1)_{2,1}| \geq \delta^{0+}$.
So, we can apply Lemma~\ref{lem: V and g is larger} again and get
$$\n_\delta(A) \geq \delta^{0+} \n_\delta(W_1 g_1 W_1 g_1 W_1)
\geq \delta^{0+} \n_\delta(W_1)^3 \geq \delta^{-3\sigma - \eps_4/2 }
= \delta^{-3 + \sigma_0 - \eps_4/2 } = \delta^{-\eps_4/2} \n_\delta(A) .$$
This contradicts \eqref{eqn: A doesnt increase}, and the proof is complete in this case.

The proof in the second case, when $V$ is not well-spread, is simpler.
Indeed, we have
$$\n_\delta(V_0) \geq \rho^\kappa \delta^{-\sigma}$$
with
$$V_0 = V \cap B_\rho(a) $$
(reusing notation).
So, by Lemma~\ref{lem: V and g is larger},
$$\n_\delta(V_1) \geq 
\delta^{0+} \n_{\delta}(V_0)^{3}
\geq \rho^{3 \kappa} \delta^{-3 \sigma+},$$
where
$$V_1 = V_0  g_0  V_0  g_0  V_0 
\subset \Gamma_{\delta^{1-}} (A_{(C)})$$
with $g_0$ from above.
By assumption on $A$, there is a finite $X \subset A$ 
so that
$$|X| \geq \rho^{-\tau}$$
and for all $x \neq x'$ in $X$,
$$\norm{x - x'} \geq C \rho.$$
Denote
$$V_2 = \bigcup_{x \in X} x V_1.$$
Therefore,
$$\n_\delta(V_2) \geq |X| \n_\delta(V_1) \geq \rho^{-\tau} \rho^{3 \kappa}
\delta^{-3\sigma+} 
\geq \rho^{-\tau/2} \delta^{-3+\sigma_0 +}
\geq \delta^{-3+\sigma_0-\eps_4 \tau / 3}
= \delta^{0-} \n_\delta(A).$$
Since $V_2 \subset \Gamma_{\delta^{1-}}(A_{(C)})$, we obtained a contradiction to
\eqref{eqn: A doesnt increase}, and the proof is complete.
\end{proof}

\end{document}